\documentclass[aos,preprint]{imsart} 

\RequirePackage[OT1]{fontenc}
\RequirePackage{amsthm,amssymb,amsfonts,amsmath}
\RequirePackage[numbers]{natbib}
\RequirePackage[colorlinks,citecolor=blue,urlcolor=blue, pdffitwindow=false, pdfstartview={FitH}]{hyperref}
\usepackage{color}
\usepackage{amsmath}
\usepackage{amssymb}
\usepackage{amsfonts}
\usepackage{amsthm}
\usepackage{rotating}
\usepackage{dsfont}
\usepackage{array}
\usepackage{multirow}
\usepackage{multicol}

\usepackage{natbib}

\usepackage{graphicx}

\usepackage{tabularx}
\usepackage{longtable}
\usepackage{lscape}
\usepackage{float}

 \usepackage{hyperref}
 \usepackage{tikz}
\usetikzlibrary{arrows}
\newtheorem{cor}{Corollary}[section]
\newtheorem{lemma}{Lemma}[section]
\newtheorem{rem}{Remark}[section]

\newtheorem{prop}{Proposition}[section]
\newtheorem{defi}{Definition}[section]

\newcommand{\E}{\mathbb{E}}
\newcommand{\R}{\mathbb{R}}

\newcommand{\sca}{SC(\alpha)}

\newcommand{\tn}{\theta_{n}}
\newcommand{\tnp}{\theta_{n+1}}
\newcommand{\tZn}{\widetilde{Z}_n}
\newcommand{\tZnp}{\widetilde{Z}_{n+1}}
\newcommand{\Znp}{Z_{n+1}}
\newcommand{\cZnp}{\check{Z}_{n+1}}
\newcommand{\cZn}{\check{Z}_{n}}
\newcommand{\htnp}{\hat{\theta}_{n+1}}
\newcommand{\htn}{\hat{\theta}_{n}}

\newcommand{\tcr}{\textcolor{black}}
\newcommand{\tcb}{\textcolor{black}}
\newcommand{\ES}{\mathbb{E}}
\newcommand{\ER}{\mathbb{R}}

\newcommand{\teps}{T_{\varepsilon,\gamma,r}}

\newtheorem{thm}{Theorem}

\arxiv{arXiv:0000.0000}

\startlocaldefs
\numberwithin{equation}{section}
\theoremstyle{plain}
\endlocaldefs

\begin{document}

\begin{frontmatter}
\title{Optimal non-asymptotic bound of the Ruppert-Polyak averaging without strong convexity}
\runtitle{Non-asymptotic analysis of the Ruppert-Polyak averaging }

\begin{aug}
\author{\fnms{S\'ebastien} \snm{Gadat}\ead[label=e1]{sebastien.gadat@math.univ-toulouse.fr}} 
\and 
\author{\fnms{Fabien} \snm{Panloup}\ead[label=e4]{fabien.panloup@math.univ-angers.fr}}

\runauthor{S. Gadat,  F. Panloup}

\affiliation{Toulouse School of Economics, University of Toulouse I Capitole}
\address{Toulouse School of Economics\\ 
 Universit\'e Toulouse 1 - Capitole.\\
21 all\'ees de Brienne\\ 31000 Toulouse, France.\\
\printead{e1}\\}

\affiliation{Laboratoire Angevin  de Recherche en Math\'ematiques, Universit\'e d'Angers}
\address{Laboratoire Angevin  de Recherche en Math\'ematiques,\\
 Universit\'e d'Angers\\
2 Boulevard Lavoisier\\
49045 Angers cedex 01, France.\\
\printead{e4}\\}

\end{aug}
\begin{abstract} 


This paper is devoted to the non-asymptotic control of the mean-squared error for the  Ruppert-Polyak  stochastic averaged gradient descent introduced in the seminal contributions of \cite{ruppert} and \cite{polyakjuditsky}. In our main results, we establish non-asymptotic tight bounds (optimal with respect to the Cramer-Rao lower bound) in a very general framework that includes the  uniformly strongly convex case as well as the one where the function $f$ to be minimized satisfies a  weaker  Kurdyka-\L ojiasewicz-type condition \cite{Lojasiewicz,Kurdyka}. In particular, it makes it possible  to recover 
some pathological examples such as  on-line learning for  logistic regression (see  \cite{BachLogit}) and  recursive quantile estimation (an even non-convex situation). 
\tcr{Finally,   our bound is optimal when the decreasing step $(\gamma_n)_{n\ge1}$ satisfies: $\gamma_n=\gamma n^{-\beta}$ with $\beta=3/4$, leading to a second-order term in $O(n^{-5/4})$.}

\end{abstract}



\end{frontmatter}

\section{Introduction}

\subsection{Averaging principle for stochastic algorithms}
\paragraph{Initial averaging algorithm}
 Let $f:\R^d\rightarrow\R $ be a function that belongs to ${\cal C}^2(\R^d,\R)$, \textit{i.e.},  the space of twice differentiable functions from $\ER^d$
 to $\ER$ with  continuous second partial derivatives. Let us assume that $\nabla f$ admits the following representation: a measurable function $\Lambda:\R^d\times\R^p\rightarrow \R^d$ and a random variable $Z$ with values in $\R^p$ exist such that $Z$ is distributed according to $\mu$ such that:
$$\forall \theta\in\R^d, \quad \nabla f (\theta)=\E[\Lambda(\theta,Z)].$$
In this case, the  Robbins-Monro procedure introduced in the seminal contribution \cite{RobbinsMonro} is built with an i.i.d. sequence of observations $(Z_i)_{i \geq 1}$ distributed according to $\mu$. It is well known  that under appropriate assumptions, the minimizers of $f$ can be approximated through the recursive stochastic algorithm $(\theta_n)_{n\ge0}$ defined by:
$\theta_0\in\R^d$ and 
\begin{equation}\label{eq:gs}
\forall n\ge0,\quad \theta_{n+1}=\theta_n-\gamma_{n+1}  \Lambda (\theta_n,Z_{n+1}),
\end{equation}
where $(\gamma_n)_{n\ge1}$ denotes a non-increasing sequence of positive numbers such that:
$$\Gamma_n:=\sum_{k=1}^n\gamma_k \xrightarrow{n \longrightarrow +\infty} +\infty\quad \text{and} \quad \gamma_n\xrightarrow{n \longrightarrow +\infty} 0.$$

The standard averaging procedure of Ruppert-Polyak (referred to as RP averaging) consists in introducing a Cesaro average over the iterations of the Robbins-Monro sequence defined by:
$$
\htn=\frac{1}{n} \sum_{k=1}^n \theta_k,\quad n\ge1.
$$
It is well known that such an averaging procedure is a way to improve the  convergence properties of the original algorithm  $(\theta_n)_{n \geq 1}$ by minimizing the asymptotic variance  induced by the algorithm. More precisely, when $f$ is a strongly convex function and possesses a unique minimum $\theta^\star$, $(\htn)_{n \geq 1}$, $ \sqrt{n}(\htn-\theta^\star)_{n\ge1}$ converges in distribution to a Gaussian law whose variance attains the Cramer-Rao lower bound of any (oracle) unbiased estimation of $\theta^\star$
  (see Theorem \ref{tcl} for a precise statement of this state of the art result).

Such results are generally achieved \textit{asymptotically} in general situations where $f$ is assumed to be strongly convex, we refer to \cite{polyakjuditsky} for the initial asymptotic description and to \cite{FortG} for some more general results. 
In \cite{BachMoulines}, a non-asymptotic optimal (with a sharp first order term) result is obtained in the strongly convex situation under restrictive moment assumptions on the noisy gradients. 
It is also dealt with non asymptotically without sharp constants in some specific cases where such a strong convexity property fails (on-line logistic regression \cite{BachLogit}, recursive median estimation \cite{Cenac,Godichon} 
for example). Nevertheless, a general result for strongly convex or not situations under some mild conditions on the noise while preserving a sharp optimal $O(n^{-1})$ rate of convergence of the $\mathbb{L}^2$-risk is yet missing.

In this paper, our purpose is to derive a sharp study on the $\mathbb{L}^2$-risk  of the averaged process $(\htn)_{n \geq 0}$ and derive an optimal variance result, which depends on the Hessian of $f$ at $\theta^\star$ without restricting ourself to the strongly convex or even convex case. To do so, we will introduce a weak assumption on $f$ that generalizes a global Kurdyka-\L ojasiewicz inequality assumption on $f$ (see \cite{Lojasiewicz,Kurdyka}).
We are also interested in the adaptivity of $(\htn)_{n \geq 0}$: the ability of the algorithm to behave optimally and independently on the local value of the Hessian of $f$ at $\theta^{\star}$. We also
alleviate the convexity conditions on $f$ under which such bounds can be achieved.

\subsection{Polyak-Juditsky central limit theorem}
To assess the quality of a non-asymptotic control of the sequences $(\htn)_{n \geq 0}$, we recall the CLT  associated with $(\htn)_{n \geq 0}$, whose statement is adapted from \cite{polyakjuditsky}\footnote{In \cite{polyakjuditsky}, the result is stated in a slightly more general framework with the help of a Lyapunov function. We have chosen to simplify the statement for the sake of readability.
} with the strongly convex assumption $\mathbf{(H_{\sca})}$ commonly used in the optimization community: \\

\noindent
\textbf{Assumption $\mathbf{H_{\sca}}$ - } \textbf{Strongly convex function}
\textit{$f$ is a strongly convex function of parameter $\alpha>0$ in the set:
\begin{equation}\label{def:sca}
\sca := \left\{ f \in \mathcal{C}^2(\R^d) \, : D^2f - \alpha I_d \geq 0 \right\}\end{equation}
where $D^2 f$ stands for the Hessian matrix of $f$ and inequality $A \geq 0$ for any matrix $A$ has to be understood in the sense of quadratic forms. }

The set $\sca$ captures many practical situations such as the least square optimization problem in statistical linear models for example.

\begin{thm}[\label{tcl}\cite{polyakjuditsky}]
Assume that:
\begin{itemize}
\item[i)] the function $f$ is in $\sca$ and $x\longmapsto D^2f (x)$ is bounded.
\item[ii)] $\gamma_n \xrightarrow{n \rightarrow + \infty} 0$ and $\gamma_n^{-1}(\gamma_n-\gamma_{n+1}) = o_{+\infty}(\gamma_n)$,
\item[iii)] the convergence in probability of the conditional covariance holds, \textit{i.e.},
$$
\lim_{n \longrightarrow + \infty} \mathbb{E}[\Delta M_{n+1} \Delta M_{n+1}^T \vert \mathcal{F}_{n}]  = S^\star,
$$
\end{itemize}   
then
$$
\sqrt{n} (\htn - \theta^\star) \xrightarrow{\mathcal{L}, n \longrightarrow + \infty} \mathcal{N}(0,\Sigma^\star),
$$
where 
\begin{equation}\label{def:Sigma2}
\Sigma^\star = \{D^2 f(\theta^{\star})\}^{-1} S^\star \{D^2 f(\theta^{\star})\}^{-1}.
\end{equation}
\end{thm}
\noindent
Theorem \ref{tcl} shows that the Ruppert-Polyak averaging produces an asymptotically optimal algorithm whose rate of convergence  is $O(n^{-1})$, which is minimax optimal in the class of strongly convex stochastic minimization problems (see, \textit{e.g.}, \cite{NemirovskiYudin}). Moreover, the asymptotic variance is also optimal because it attains the Cramer-Rao lower bound (see \textit{e.g.} \cite{polyakjuditsky,Casella}).

It is also important to observe that $(\htn)_{n \geq 0}$ is an adaptive sequence since the previous result is obtained independently of the size of $D^2 f(\theta^\star)$ as soon as the sequence $(\gamma_n)_{n  \geq 1}$ is chosen as $\gamma_n = \gamma n^{-\beta}$ with $\beta \in (0,1)$. 

\subsection{Main results}

As pointed out by many authors in some recent works (we refer to \cite{BachMoulines}, \cite{BachLogit} and \cite{Cenac},  among others), even though very general,   Theorem \ref{tcl} has the usual drawback of being only asymptotic with respect to $n$. To bypass this weak quantitative result, some improvements are then obtained for various particular cases of minimization problems (\textit{e.g.}, logistic regression, least square minimization, median and quantile estimations) in the literature. 

Below, we are interested in deriving some non-asymptotic inequality results on the RP averaging for the minimization of $f$. We therefore establish some non-asymptotic mean squared error upper bounds for $(|\htn-\theta^\star|_2^2)_{n \geq 1}$. We  also investigate some more specific situations without any strong convexity property (quantile estimation and logistic regression). In each case, we are interested in the first-order term of the limiting variance involved in Theorem \ref{tcl} (and in the Cramer-Rao lower bound as well).

\subsubsection{Notations}

The canonical filtration $(\mathcal{F}_n)_{n \geq 1}$ refers to the filtration associated with the past events before time $n+1$: $\mathcal{F}_n = \sigma(\theta_1,\ldots,\theta_n)$ for any $n \geq 1$. The conditional expectation at time $n$ is then denoted by $\mathbb{E}[ . \vert \mathcal{F}_n]$.

For any vector $y \in \R^d$, $y^T$ denotes the transpose of $y$, whereas $|y|$ is the Euclidean norm of $y$ in $\R^d$. The set $\mathcal{M}_{d}(\R)$ refers to the set of squared real matrices of size $d \times d$ and the tensor product $\otimes 2$ is used to refer to the following  quadratic form:
$$
\forall M \in \mathcal{M}_{d}(\R) \quad \forall y \in \R^d \qquad M y^{\otimes 2} = y^T M y.
$$
$I_d$ is the identity matrix in $\mathcal{M}_{d}(\R)$ and  $\mathcal{O}_d(\R)$ denotes the set of orthonormal real matrices of size $d \times d$:
$$
\mathcal{O}_d(\R) := \left\{ Q \in \mathcal{M}_d(\R) \, : \, Q^T Q = I_d\right\}.
$$
Finally, the notation $\|\, .\,\|$ corresponds to a (non-specified) norm on $\mathcal{M}_{d}(\R)$. 

For two positive sequences $(a_n)_{n \geq 1}$ and $(b_n)_{n \geq 1}$, the notation $a_n \lesssim b_n$ refers to the {domination relationship, $i.e.$ $a_n \leq c \, b_n$ where $c$ is independent of $n$. The binary relationship $a_n = \mathcal{O}(b_n)$ then holds if and only if $|a_n| \lesssim |b_n|$. Finally, if for all $n\in\mathbb{N}$, $b_n\neq0$, $a_n = o(b_n)$ if $ \lim  \frac{a_n}{b_n} = 0$ when $n \longrightarrow +\infty$.\medskip

In the rest of the paper, we  assume that  $f$ satisfies the following properties:
\begin{equation}\label{condbase}
\lim_{|x|\rightarrow+\infty} f(x)=+\infty\quad \textnormal{and}\quad \{x\in\ER^d,\nabla f(x)=0\}=\{\theta^\star\}
\end{equation} 
where $\theta^\star$ is thus the unique minimum of $f$. Without loss of generality, we also assume that $f(\theta^\star)=0$. 


We also consider the common choice for $(\gamma_n)_{n \geq 1}$ (for $\gamma>0$ and $\beta \in (0,1)$):
$$
\forall n \geq 1 \qquad \gamma_n = \gamma n^{-\beta}.
$$
In particular, we have $\Gamma_n \sim \frac{\gamma}{1-\beta} n^{1-\beta} \longrightarrow + \infty$ and $\gamma_n \longrightarrow 0$ as $n \longrightarrow + \infty$.\bigskip

The rest of this section is devoted to the statement of our main results. In Subsection \ref{sec:main1}, we state our main general result (Theorem \ref{thm:avgd}) under some general assumptions on the noise part and on the behavior of the $L^p$-norm of the \textbf{original procedure} $(\theta_n)_{n \geq 1}$  (\textit{$(L^{p}$,$\sqrt{\gamma_n})$-consistency}). Then, in the next subsections, we provide some settings where this consistency condition is satisfied: under a strong convexity assumption in Subsection \ref{sec:main2} and under a Kurdyka-\L ojiasewicz-type assumption  in Subsection \ref{sec:main3}.
\subsubsection{Non asymptotic adaptive and optimal inequality}\label{sec:main1}
Our first main result is Theorem \ref{thm:avgd} and we introduce  the following definition.

\begin{defi}[$(L^{p}$,$\sqrt{\gamma_n})$-consistency]\label{defi:Lp}
Let $p>0$. We say that a sequence $(\tn)_{n \geq 1}$ satisfies the $(L^{p}$,$\sqrt{\gamma_n})$-consistency (convergence rate condition) if
$\left(\frac{\tn}{\sqrt{\gamma_n}}\right)_{n\ge1}$ is bounded in $L^p$, $i.e.$, if:
$$
\exists \, c_p >0 \quad \forall n \geq 1 \qquad  \mathbb{E} |\tn|^{p} \leq c_p \{\gamma_n\}^\frac{p}{2}.
$$
\end{defi}
\noindent 
Note that according to the Jensen inequality, the $(L^{p}$,$\sqrt{\gamma_n})$-consistency implies the $(L^{q}$,$\sqrt{\gamma_n})$-consistency for any $0<q<p$. As mentioned before,
this definition refers to the behaviour of the crude procedure $(\tn)_{n \geq 1}$ defined by Equation \eqref{eq:gs}.  We will prove that Definition \ref{defi:Lp} is a key property to derive non-asymptotic bounds for the RP-averaged algorithm $(\htn)_{n\ge1}$ (see Theorem  \ref{thm:avgd} below).\smallskip
%
%
%

We also introduce a smoothness assumption on the covariance of the martingale increment:\\

\noindent
\textbf{Assumption $\mathbf{(H_S)}$ - } \textbf{Covariance of the martingale increment}
\textit{The covariance of the martingale increment satisfies:
$$
\forall \theta \in \mathbb{R}^d \qquad
\ES \left[ \Delta M_{n+1} \Delta M_{n+1}^t \vert \mathcal{F}_n \right] =
 S(\theta_n) \qquad a.s.
$$
where $S:\ER^d\rightarrow{\cal M}_d(\ER)$ is a Lipschitz continuous function:
$$
\exists L >0 \quad 
\forall (\theta_1,\theta_2) \in \mathbb{R}^d \qquad \|S(\theta_1)-S(\theta_2)\| \leq L |\theta_1-\theta_2|.
$$}

When compared
 to   Theorem \ref{tcl},  Assumption  $\mathbf{(H_S)}$ is more restrictive but in fact corresponds to the usual framework. 
Under additional technicalities, this assumption may be relaxed to a local Lipschitz behaviour of $S$. For reasons of clarity, we preferred to reduce our purpose to this reasonable setting.

 

We now state our main general result:

\begin{thm}[Optimal non-asymptotic bound for the averaging procedure\label{thm:avgd}]
 Let $\gamma_n=\gamma n^{-\beta}$ with $\beta \in (1/2,1)$. Assume that $(\tn)_{n \geq 1}$ is $(L^4$,$\sqrt{\gamma_n})$-consistent and that Assumption $(\mathbf{H}_{S})$ holds. Suppose moreover that $D^2f (\theta^\star)$ is  \tcr{positive-definite}.Then, a large enough $C$  exists such that:
$$
\forall n \in \mathbb{N}^\star \qquad 
\mathbb{E} \left[|\htn-\theta^\star|^2 \right] \leq \frac{{\rm Tr}(\Sigma^\star)}{n}+C n^{-r_\beta},
$$
where $\Sigma^\star$ is defined in Equation \eqref{def:Sigma2} (with $S^\star=S(\theta^\star)$) and 
$$r_\beta= \left(\beta+\frac 1 2 \right) \wedge \left(2-\beta\right).$$ In particular, $r_\beta>1$  for all $\beta\in\left({1}/{2},1\right)$ and $\beta\longmapsto r_\beta$ attains its maximum for $\beta=3/4$, which yields:
$$
\forall n \in \mathbb{N}^\star \qquad 
\mathbb{E} \left[|\htn-\theta^\star|^2 \right] \leq \frac{{\rm Tr}(\Sigma^\star)}{n}+C n^{-5/4}.
$$
\end{thm}

 \smallskip

The result stated by Theorem \ref{thm:avgd} deserves three main remarks. \smallskip

$\bullet$ First, we obtain the exact optimal rate $O(n^{-1})$ with the sharp constant ${\rm Tr}(\Sigma^\star)$ as shown by Theorem \ref{tcl}. Hence, at the first order, Theorem \ref{thm:avgd} shows that the averaging procedure is minimax optimal with respect to the Cramer-Rao lower bound. Moreover, the result is adaptive with respect to the value of the Hessian $D^2 f(\theta^\star)$: any sequence $\gamma_n = \gamma n^{-\beta}$ with $\beta \in (1/2,1)$ and $\gamma>0$, regardless the value of $\beta$ or $\gamma$, produces the result of Theorem \ref{thm:avgd}. We should note that such an adaptive property does not hold for the initial sequence $(\tn)_{n \geq 1}$ as proved by the central limit theorem satisfied by $(\tn)_{n \geq 1}$ (see \cite{Duflo} for example). To the best of our knowledge, such a result \tcr{was} only obtained in \cite{BachMoulines} for strongly convex objective functions. \smallskip

$\bullet$ Second, Theorem \ref{thm:avgd} does not require any convexity \tcr{assumptions} \tcb{on} $f$.
 However, this formulation is misleading by itself since we instead assume a  $(L^{p}$,$\sqrt{\gamma_n})$-consistency for the initial sequence $(\tn)_{n \geq 1}$. We will discuss  how we can guarantee such a property in Theorem \ref{thm:sgd} in convex situations or in some more general cases. The conclusion of Theorem \ref{thm:avgd} holds as soon as $(\tn)_{n \geq 1}$ satisfies $\mathbb{E}[|\tn-\theta^\star|^4] \leq C \{\gamma_n\}^2$, which permits us to efficiently linearize the drift of the sequence $(\htn)_{n \geq 1}$.
 
 Our proof is quite different from the one of \cite{BachMoulines}, which gives an optimal result only in the strongly convex case, whereas the rate is seriously damaged in the non-strongly convex situation (Theorems 3 and  6 of \cite{BachMoulines}). In contrast, our result can also apply to some non-strongly convex objective functions while preserving the optimal upper bound, and holds under much weaker conditions on the noise setting.
 Our improvement is achieved through  a spectral analysis of the  second-order Markov chain induced by  $(\htn)_{n \geq 1}$. This spectral analysis requires a preliminary linearization step of the drift from $\htn$ to $\htnp$. The cost of this linearization is absorbed by a preliminary control of the initial sequence $(\tn)_{n \geq 1}$, $(L^{p}$,$\sqrt{\gamma_n})$-consistency  for $p=4$ (see Proposition \ref{prop:sgdsc} and  Theorem \ref{theo:weaksgd}  for results on the $(L^{p}$,$\sqrt{\gamma_n})$-consistency).
Note that this $(L^{p}$,$\sqrt{\gamma_n})$-consistency  for $p=4$ is also a property used in \cite{Godichon} and \cite{BachMoulines}.

$\bullet$   Finally, we prove that the second order term is $O(n^{-r_\beta})$ and that its size is minimized according to the choice $\beta=3/4$. With this optimal calibration of 
$\beta$, the size of the second-order term is $n^{-5/4}$ \tcr{in the general case}. \tcr{(As mentioned in Remark \ref{rem:linearcase}, this bound can be improved if the third derivative of $f$ vanishes in $\theta^\star$}. In the literature, several choices for the value of $\beta$ have been proposed. In the particular case of the recursive quantile problem (see \textit{e.g.} \cite{Godichon}), the authors suggest to use $\beta=2/3$ to minimize the second-order terms without any explicit quantitative result.
For strongly convex functions, it is indicated in \cite{BachMoulines} to also use $\beta=2/3$ and the second order term obtained in \cite{BachMoulines} is $n^{-7/6}$, which is larger than $n^{-5/4}$. Even though the second-order terms are of  marginal importance, Theorem \ref{thm:avgd} provides stronger results than Theorem 3 of \cite{BachMoulines} and results stated in \cite{Cenac,Godichon}.
 It also appears  that the condition $\beta>1/2$ and $\beta<1$ is necessary to obtain the tight constant ${\rm Tr}(\Sigma^\star)$, and the choice $\beta=1/2$ does not seem appropriate in a general situation according to what is claimed in \cite{BachMoulines} and contrary to what is claimed in \cite{BachLogit} in the particular case of the on-line logistic regression.

\subsubsection{$(L^{p}$,$\sqrt{\gamma_n})$-consistency with strong convexity}\label{sec:main2}
In this section, we temporarily restrict our study to the classical setting  $\mathbf{H_{\sca}}$ and we need to add an additional condition on the noise, denoted by $\mathbf{(H^{SC}_{\Sigma_p})}$:


\noindent 
\textbf{Assumption $\mathbf{(H^{SC}_{\Sigma_p})}$ - }\textbf{Moments of the martingale increment}
\textit{For a given $p \in \mathbb{N}^\star$, the sequence of martingale increments satisfies: a constant $\Sigma_p$ exists such that for any $n\in\mathbb{N}$:
$$
\ES[|\Delta M_n|^{2p}|{\cal F}_n]\leq \Sigma_p (1+(f(\tn))^{p} \qquad \text{a.s.}
$$ }

We emphasize that even though Assumption $\mathbf{H_{\sca}}$ is a potentially restrictive assumption on $f$, the one on the martingale increments is not restrictive and allows 
a polynomial dependency in $f(\tn)$ of the moments of $\Delta M_n$, which is much weaker than the one used in Theorem 3 of \cite{BachMoulines}.
For example, such  an assumption holds  in the case of the recursive linear least square problem. In that case, we retrieve the baseline assumption introduced in \cite{Duflo} that only provides an almost sure convergence of $(\tn)_{n \geq 1}$ towards $\theta^\star$ without any rate.
%
In this setting, we can state the following proposition, whose proof is left to the reader and up to some minor modifications, is contained in the more general result stated in Theorem \ref{thm:sgd} (see Section \ref{sec:main3}).
\begin{prop}\label{prop:sgdsc}
Assume that $\alpha >0$ exists such that $f$ is $\mathbf{H_{\sca}}$ and that  $x \longmapsto D^2 f(x)$ is Lipschitz bounded. Then, if the sequence $(\Delta M_n)_{n \geq 1}$ satisfies $\mathbf{(H^{SC}_{\Sigma_p})}$, then $(\tn)_{n \geq 1}$ is $(L^{p}$,$\sqrt{\gamma_n})$-consistent: a constant $C_p$ exists such that:
$$
\mathbb{E} |\tn-\theta^\star|^{p} \leq C_p \{\gamma_n\}^{p/2}.
$$
\end{prop}
An immediate consequence of Proposition \ref{prop:sgdsc} and of Theorem \ref{thm:avgd}
on the sequence $(\htn)_{n \geq 1}$ is given by the next corollary.
\begin{cor}\label{cor:sca}
Assume that $\gamma_n=\gamma n^{-\beta}$ with $\beta \in (1/2,1)$. Then, if  $\mathbf{(H_S)}$  and  the assumptions of Proposition \ref{prop:sgdsc} hold,  we have:
 $$
\forall n \in \mathbb{N}^\star \qquad 
\mathbb{E} \left[|\htn-\theta^\star|^2 \right] \leq \frac{{\rm Tr}(\Sigma^\star)}{n}+C n^{-r_\beta}
$$
where $r_\beta$ is defined in Theorem \ref{thm:avgd}.
\end{cor}

\subsubsection{$(L^{p}$,$\sqrt{\gamma_n})$-consistency without strong convexity}\label{sec:main3}
 In some many interesting cases, the latter strongly convex Assumption $\mathbf{H_{\sca}}$  does not hold because the repelling effect towards $\theta^\star$ of $\nabla f(x)$ is not strong enough for large values of $|x|$. For example, this is the case in the logistic regression problem or in the recursive quantile estimation where the function $\nabla f$ is asymptotically flat for large values of $|x|$.
Motivated by these examples, we thus aim to  generalize the class of functions $f$ for which the $(L^{p}$,$\sqrt{\gamma_n})$-consistency property holds.
For this purpose, we introduce Assumption $(\mathbf{H_\phi})$ defined by: \smallskip

\noindent
\textbf{Assumption $(\mathbf{H_\phi})$ - }\textbf{Weakly reverting drift}
\textit{The function $f$ is $\mathcal{C}^2(\ER^d,\ER)$ with $D^2f$ bounded and Lipschitz with $D^2 f(\theta^\star)$ invertible and}
\begin{itemize} 
\item \textit{$i)$ $\phi$ is  ${\cal C}^2(\ER_+,\ER_+)$ non-decreasing and  $\exists \, x_0\ge0 \, : \forall x\ge x_0$, $\phi''(x)\le 0$. }

\item \textit{ $ii)$ Two positive numbers  $m$ and $M$ exist such that $\forall x\in\ER^d\backslash \{\theta^\star\}$:
\begin{equation}\label{eq:hphi}
0< m\le \phi'(f(x))|\nabla f(x)|^2+\frac{|\nabla f(x)|^2}{f(x)}\le M.
\end{equation}}
\end{itemize}
 
 Roughly speaking, the function $\phi$ quantifies the lack of convexity far from $\theta^\star$ and is calibrated in such a way that the function $x\rightarrow f^p(x) e^{\phi( f(x))}$ is strongly convex. The extremal situations are the following ones: when $\phi\equiv 1$,  we recover the previous case or more precisely, when $x\longmapsto D^2 f(x)$ is Lipschitz continuous, $(\mathbf{H_{\sca}})\Longrightarrow (\mathbf{H_\phi})$ with $\phi\equiv 1$. Actually, in this case, it is straightforward to prove that some positive constants $c_1$ and $c_2$ exist such that for all $x\in \ER^d$,
$$ c_{1} |x-\theta^\star|^2\le f(x)\le c_{2} |x-\theta^\star|^2,\quad \textnormal{and}\quad c_{1} |x-\theta^\star|\le |\nabla f(x)|\le c_{2} |x-\theta^\star|.$$
Note that $(\mathbf{H_\phi})$ remains slightly more general since it even can be true in some cases where $D^2 f$ is not strictly positive everywhere.\smallskip

The opposite case is $\phi(x)=x$. In this setting,  $(\mathbf{H_\phi})$
 is satisfied when $m\le  |\nabla f(x)|^2\le M$ with some positive $m$ and $M$. Note that this framework \tcr{includes} the online logistic regression and the recursive quantile estimation (see Subsection \ref{sec:main5}).\smallskip

For  practical purposes, we introduce below a kind of parametric version of Assumption $(\mathbf{H_\phi})$ denoted by $(\mathbf{H_{KL}^r})$, which may be seen as a 
global \textit{Kurdyka-\L ojasiewicz  gradient inequality} (see, \textit{e.g.}, \cite{Kurdyka,Lojasiewicz}  and Subsection \ref{sec:hphi} for details): 


\noindent
\textbf{Assumption $(\mathbf{H_{KL}^r})$ - }\textbf{Global KL inequality}
\textit{The function $f$ is $\mathcal{C}^2(\ER^d,\ER)$ with $D^2f$ bounded and Lipschitz with $D^2 f(\theta^\star)$ invertible and}
\begin{itemize} 
\item \textit{For $r \in [0,1/2]$, we have
\begin{equation}
\liminf_{|x| \longrightarrow + \infty} f^{-r} |\nabla f| >0 \qquad \text{and} \qquad \limsup_{|x| \longrightarrow + \infty} f^{-r} |\nabla f| >0
\end{equation}}
\end{itemize}
 $(\mathbf{H_\phi})$ and $(\mathbf{H_{KL}^r})$ are linked by the following lemma:
 \begin{prop}\label{lemma:KLimphi} Let $r\in[0,1/2]$ such that $(\mathbf{H_{KL}^r})$ holds. Then, $(\mathbf{H_\phi})$ holds with $\phi$ defined by $\phi(x)=(1+|x|^2)^{\frac{1-2r}{2}}$.
 Furthermore,
 \begin{equation}\label{eq:polynincrKL}
 \liminf_{|x|\rightarrow+\infty} f(x) |x|^{-\frac{1}{1-r}} >0.
\end{equation}
  \end{prop}
The implication is easy to prove (using that near $\theta^\star$, $f(x)\lesssim |x-\theta^\star|^2$ and    $ |x-\theta^\star|\lesssim|\nabla f(x)|$ since  $\nabla f( \theta^\star)=0$ and $D^2 f(\theta^\star)$ is strictly positive). The proof of the more intricate property  \eqref{eq:polynincrKL} is postponed to Appendix \ref{sec:appendixAA}. Note that this property will be important to 
derive the $(L^{p}$,$\sqrt{\gamma_n})$-consistency (see Theorem \ref{theo:weaksgd}). As mentioned before, further comments on these assumptions are postponed to 
Subsection \ref{sec:hphi} and the rest of this paragraph is devoted to the main corresponding results.\smallskip



As in the strongly convex case, Assumptions $(\mathbf{H_\phi})$ and $(\mathbf{H_{KL}^r})$  certainly need to be combined with some assumption on the martingale increment. As 
one might expect, the condition is unfortunately (much) more stringent than in the strongly convex case: 


\noindent
\textbf{Assumption $\mathbf{(H^{\phi}_{\Sigma_p})}$ - }\textbf{ Moments of the martingale increment}
A locally bounded deterministic function  $\rho_p:\ER_+\mapsto \ER_+$ exists such that:
\begin{equation}\label{eq:noisephi}
\forall u \ge 0 \qquad \ES[|\Delta M_n|^{2p+2}e^{\phi(u |\Delta M_n|^2)}|{\cal F}_n]\le \rho_p(u) \qquad \text{a.s.}\end{equation}

\begin{rem} The general form of this assumption can be roughly explained as follows: one of the main ideas of the proof of Theorem \ref{theo:weaksgd} below is to use the function $x\mapsto f^p(x)^p e^{\phi(f(x))}$ as a Lyapunov-type function  in order to obtain some contraction properties.
Note that when  $(\Delta M_n)_{n\ge1}$ is a bounded sequence, $\mathbf{(H^{\phi}_{\Sigma_p})}$ is automatically satisfied (this is the case for the quantile recursive estimation and for the logistic regression of bounded variables: see Subsection \ref{sec:main5}).

However, when $\phi\equiv 1$ (\textit{i.e.} strongly convex case), it can be observed  that   $\mathbf{(H^{SC}_{\Sigma_p})}$ is not retrieved as it would have been expected. This can be explained 
by the fact that Assumption $\mathbf{(H^{\phi}_{\Sigma_p})}$ is adapted to the general case and that the particular case $\phi\equiv 1$, certainly leads to some simplifications (especially in the derivation of the Lyapunov function).
Nevertheless, we could (with additional technicalities) also allow a dependency  in $f(\tn)$ by replacing the right-hand member of the assumption with   $C(1+(f(\theta_n))^{p-1}$. However, this
 seems of limited interest  in the general case in view of  the exponential term of the left-hand side. More precisely, the dependency in $f(\theta_n)$ could be really
 interesting for applications if it were of comparable size to the left-hand member. \tcr{Finally, let us remark that as it can be expected, the constraint on the noise increases with $\phi$,
 $i.e.$, with the lack of convexity of the function $f$.}
\end{rem}

We then state the main result of this paragraph that holds in a generic \textcolor{black}{potentially} non-convex situation supported by  $\mathbf{(H_{\phi})}$.

\begin{thm}\label{theo:weaksgd}
For any $p\ge 1$:
\begin{itemize}
\item[$i)$] Assume that $f$ satisfies $\mathbf{(H_{\phi})}$  and that the martingale increment sequence satisfies $\mathbf{(H^{\phi}_{\Sigma_p})}$, then a constant $C_p$ exists such that:
$$
\mathbb{E} [ f^p(\tn) e^{\phi(f(\tn))}] \leq C_p \{\gamma_n\}^{p}.
$$
\item[$ii)$] If, furthermore, $\liminf_{|x|\rightarrow+\infty} |x|^{-2p} f^p(x)e^{\phi(f(x))}>0,$
then $(\tn)_{n \geq 1}$ is $(L^{2p}$,$\sqrt{\gamma_n})$-consistent: a constant $C_p$ exists such that:
$$
\mathbb{E} |\tn-\theta^\star|^{2p} \leq C_p \{\gamma_n\}^{p}.
$$
\item[$iii)$]
 In particular,  $(\tn)_{n \geq 1}$ is $(L^{2p}$,$\sqrt{\gamma_n})$-consistent if $\mathbf{(H_{KL}^r)}$ holds for a given $r \in [0,1/2]$ and $\mathbf{(H^{\phi}_{\Sigma_p})}$ holds with $\phi(t) = (1+t^2)^{(1-2r)/2}$.
\end{itemize}
\end{thm}
}
\begin{proof}
The proof of Theorem \ref{theo:weaksgd} $i)$ is postponed to Section \ref{sec:proofsgd}. 

The second statement $ii)$ is a simple consequence of $i)$: actually, we only need to prove that the function $\tau$ defined by
$\tau(x)=f^p(x) e^{\phi(f(x)})$, $x\in\ER^d$,  satisfies $\inf_{x\in\ER^d\backslash\{0\}}\tau(x)|x-\theta^\star|^{-2p}>0$.
 Near $\theta^\star$,  the fact that $D^2f(\theta^\star)$ is  \tcr{positive-definite} (see Subsection \ref{sec:hphi} for comments on this property) can be used to ensure that $x\mapsto \tau(x)|x|^{-2p}$ is lower-bounded by a positive constant. Then, since $\tau$ is positive on $\ER^d$, the result follows from the additional assumption $\liminf_{|x|\rightarrow+\infty} \tau(x)|x|^{-2p}>0$. 
\smallskip

Finally,  for $iii)$, we only have to prove that the additional statement of $ii)$ holds  under $\mathbf{(H_{KL}^r)}$. This point is a straightforward consequence of \eqref{eq:polynincrKL}
and of the fact that $\phi(x)=(1+|x|^2)^{\frac{1-2r}{2}}$ in this case.
\end{proof}

Applying Theorem \ref{thm:avgd} makes it possible  to derive non-asymptotic bounds under $\mathbf{(H_{\phi})}$. We chose to only state the result under the parametric assumption $\mathbf{(H_{KL}^r)}$.

\begin{cor}\label{cor:ratehphi}
Assume $\mathbf{(H_S)}$, $\mathbf{(H_{KL}^r)}$ and   $\mathbf{(H^{\phi}_{\Sigma_p})}$ with $p=2$, $r \in [0,1/2]$ and {\small $\phi(t) = (1+t^2)^{\frac{1-2r}{2}}$}. If $\gamma_n=\gamma n^{-\beta}$ with $\beta\in(1/2,1)$, then $(\htn)_{n \geq 1}$ satisfies:
 $$
\forall n \in \mathbb{N}^\star \qquad 
\mathbb{E} \left[|\htn-\theta^\star|^2 \right] \leq \frac{{\rm Tr}(\Sigma^\star)}{n}+C n^{-r_\beta},
$$
where $r_\beta$ is defined in Theorem \ref{thm:avgd}.
\end{cor}

  \begin{rem}
At   first sight, the result brought by Corollary \eqref{cor:ratehphi}  may appear surprising since we obtain a \tcb{$O(1/n)$} rate for the mean-squared error of the averaged sequence towards $\theta^\star$  \emph{without strong convexity}, including, for example, some situations where $f(x)\sim|x|$ as $|x|\rightarrow+\infty$.  This could be \tcr{viewed as} a contradiction with the minimax rate of convergence $O(1/\sqrt{n})$ for stochastic optimization problems in the simple convex case (see, \textit{e.g.}, \cite{NemirovskiYudin} or \cite{ABRW}). The above minimax result simply refers to the worst situation in the class of convex functions \emph{that are not necessarily differentiable}, whereas Assumption $\mathbf{(H_{\phi})}$ used in Corollary \ref{cor:ratehphi} describes a set of functions that are not necessarily strongly convex or even simply convex, but all the functions involved in $\mathbf{(H_{\phi})}$ or in $(\mathbf{H_{KL}^r})$  belong to $\mathcal{C}^2(\R^d,\R)$. In particular, the worst case  is attained in \cite{ABRW} through linear combinations of shifted piecewise affine functions $x \longmapsto |x+1/2|$ and $x \longmapsto |x-1/2|$, functions for which Assumption $\mathbf{(H_{\phi})}$ is obviously not satisfied.

\end{rem}
\subsection{Comments on Assumption $(\mathbf{H_\phi})$ and link with the Kurdyka-\L ojasiewicz inequality }\label{sec:hphi}

To the best of our knowledge, this assumption is not standard in the stochastic optimization literature and thus deserves several comments, included in this section.
For this purpose, for any symmetric real matrix $A$, let us denote  the lowest eigvenvalue of $A$ by $\underline{\lambda}_A$.

\paragraph{$f$ does not necessarily need to be convex}
It is important to notice  that the function $f$ itself is not necessarily assumed to be convex under Assumption $(\mathbf{H_\phi})$. The minimal requirement is that $f$ only possesses a unique critical point (minimum). Of course, our analysis will still be based on a descent lemma for the sequences $(\tn)_{n \geq 0}$. Nevertheless, we will use a Lyapunov analysis that will involve $f^{p} e^{\phi(f)}$ instead of $f$ itself for the sequence $(\tn)_{n \geq 0}$. The descent property will then be derived from Equation \eqref{eq:hphi} in $ii)$ of $(\mathbf{H_\phi})$. Thereafter, we will be able to exploit a spectral analysis of the dynamical system that governs $(\htn)_{n \geq 0}$. 
We stress the fact that, in general, the results without any convexity assumption on $f$ are usually limited to almost sure convergence  with the help of the Robbins-Siegmund Lemma (see, \textit{e.g.}, \cite{Duflo} and the references therein). As will be shown later on, Assumption $(\mathbf{H_\phi})$ will be sufficient to derive efficient convergence rates for the averaged sequence $(\htn)_{n \geq 0}$ without any \tcr{strong} convexity assumption.

\paragraph{$f$ is necessarily a sub-quadratic and $L$-smooth function}

Let us first remark that $(\mathbf{H_\phi})$ entails an a priori upper bound for $f$ that cannot increase faster than a quadratic form. We have:
\begin{eqnarray*}
\forall x \in \ER^d \qquad \frac{|\nabla f(x)|^2}{f(x)}\le M  & \Longrightarrow & \vert \nabla (\sqrt{f}) \vert \leq \frac{\sqrt{M}}{2} \\
& \Longrightarrow & f(x) \leq \frac{M}{4} \|x\|^2.
\end{eqnarray*}  
However, we also need a slightly stronger condition with $D^2 f$ bounded over $\R^d$, meaning that $f$ is $L$-smooth for a suitable value of $L$ (with an $L$-Lipschitz gradient). We refer to \cite{Nesterov2} for a general introduction to this class of functions. Even in  the deterministic setting, the $L$-smooth property is a common minimal requirement for obtaining  a good convergence rate for smooth optimization problems, since it makes it possible to produce a descent lemma result (see, \textit{e.g.}, \cite{Bertsekas}).

%
%
%

\paragraph{About the Kurdyka-\L ojasiewicz inequality} 
As mentionned before

It is important to note that $(\mathbf{H_\phi})$ should be related to   the Kurdyka-\L ojasiewicz gradient inequalities. In the deterministic setting, the \L ojasiewicz gradient inequality \cite{Lojasiewicz} with  exponent $r$ may be stated as follows:
\begin{equation}\label{eq:KL}
\exists m >0 \quad \exists \, r \in [0,1) \quad \forall x \in \ER^d \qquad f(x)^{-r} |\nabla f(x)| \geq m,
\end{equation}
while a generalization (see, \textit{e.g.}, \cite{Kurdyka}) is governed by the existence of a concave increasing ``desingularizing" function $\psi$ such that:
$$
|\nabla(\psi \circ f)|\geq 1.
$$
The \L ojasiewicz gradient inequality is then just a particular case of the previous inequality while choosing $\psi(t) = c t^{1-r}.$ We refer to \cite{Bolte} for a recent work on how to characterize some large families of functions $f$ such that a generalized KL-inequality holds. \vspace{1em}

In this paper, the  Kurdyka-\L ojasiewicz-type gradient inequality appears through Assumption $\mathbf{(H_{KL}^r)}$ with $r\in[0,1/2]$, which implies $\mathbf{(H_\phi)}$ (see Proposition \ref{lemma:KLimphi}).
 However,   it should be noted that Assumption $\mathbf{(H_{KL}^r)}$  is slightly different from \eqref{eq:KL} since we only enforce the function $f^{-r}|\nabla f|$ to be \emph{asymptotically}
lower-bounded by a positive constant. \\

Nevertheless, in our setting where $f$ has only one critical point and where $D^2f(\theta^\star)$ is positive-definite, it is easy to prove that 
$\mathbf{(H_{KL}^r)}$ implies \eqref{eq:KL}.  Indeed, around $\theta^\star$,  $D^2 f(\theta^\star)$ is positive definite so that we could choose $r=1/2$ and then satisfy the \L ojasiewicz gradient inequality \eqref{eq:KL} on the neighborhood of $\theta^\star$.  Hence, the link between $\mathbf{(H_{KL}^r)}$  and \eqref{eq:KL}  has to be understood for large values of $|x|$.\\

 Moreover, Proposition \ref{lemma:KLimphi} states that the classical \L ojasiewicz gradient inequality \eqref{eq:KL}   associated with the assumption of the \textbf{local} invertibility of $D^2 f(\theta^\star)$ implies Assumption $(\mathbf{H_\phi})$.
The choice $r=1/2$ in Equation \eqref{eq:KL} corresponds to   the strongly-convex case  with $\phi=1$ and $\psi(t)=\sqrt{t}$.
  Conversely, the \L ojasiewicz exponent $r=0$ corresponds to the weak repelling force $|\nabla f(x)|^2\propto 1$ as $|x|\rightarrow+\infty$ and $\phi(t)=\sqrt{1+t^2}$, leading to $\psi(t)=t$.\\

 Finally, we can observe that the interest of Assumption $(\mathbf{H_\phi})$ in the stochastic framework is   more closely related to the behavior of the algorithm  when $(\tn)_{n \geq 1}$ is far away from the target point $\theta^\star$, whereas in the deterministic framework, the main interest of the desingularizing function $\psi$ is used around $\theta^\star$ to derive fast linear rates even in non strongly convex situations. For example, \cite{Bolte2} established exponential convergence of the forward-backward splitting FISTA  to solve the Lasso problem with the help of KL inequalities although the minimization problem is not strongly convex and the core of the study is the understanding of the algorithm near $\theta^\star$.
  In simple terms, the difficulty to assert some good properties of stochastic algorithms is not exactly the same as the one for deterministic problems: it is  much more difficult to control the time for a stochastic algorithm to come back far away from $\theta^\star$ than for a deterministic method with a weakly reverting effect of $- \nabla f$ because of the noise on the algorithm. In contrast, the rate of a deterministic method crucially depends on the local behavior of $\nabla f$ around $\theta^\star$.
  
\paragraph{Counter-examples of the global KL inequality} Finally, we should have in mind what kind of functions do not satisfy the global   \L ojasiewicz gradient inequality given in Equation \eqref{eq:KL}. Since we assumed $f$ to have a unique minimizer $\theta^\star$ with $D^2 f(\theta^\star)$ invertible, Inequality $ f^{-r}  |\nabla f| \geq m >0$ should only fail asymptotically. {From Equation  \eqref{eq:polynincrKL} of Proposition \ref{lemma:KLimphi}, we know that $|x|\lesssim f(x)$ for large values of $|x|$.
As a consequence, any function $f$ with logarithmic growth or comparable to $|x|^r$ growth with $r\in(0,1)$  at infinity can not be managed by this assumption.} {In the deterministic case, it is nevertheless possible to handle some less increasing functions since deterministic gradient methods are essentially affected by the local behaviour of $f$ around $\theta^\star$ under a local KL-inequality (see, \textit{e.g.}, \cite{Bolte2}). In our stochastic setting, the effect of the noise when the algorithm is far from the target point $\theta^\star$ is not negligible in comparison with the local behaviour of the algorithm near $\theta^\star$.}


Another counter-example of $f$ occurs when $f$ exhibits an infinite sequence of oscillations in the values of $f' \geq 0$ with longer and longer areas near $f'=0$ when $|x|$ is increasing. We refer to \cite{Bolte2} for the following function that do not satisfy the KL inequality for any \tcb{$r \geq 2$}:
$$
f:  x \longrightarrow x^{2r} [2+\cos ( x^{-1})] \quad \text{if} \quad x \neq 0 \quad \text{and} \quad f(0)=0.
$$


\subsection{Applications}\label{sec:main5}

\paragraph{Strongly convex situation}
First, we  can observe that in the strongly convex situation, Corollary \ref{cor:sca} provides a very tractable criterion to assess the non-asymptotic first-order optimality of the averaging procedure since $\mathbf{(H^{SC}_{\Sigma_p})}$ is very easy to check.

For example, considering the \textbf{stochastic recursive least mean square estimation} problem (see, \textit{i.e.}, \cite{Duflo}), it can immediately be checked that $\theta \longrightarrow f(\theta)$ is quadratic. In that case, the problem is strongly convex, and the noise increment  satisfies:
$$
\ES[|\Delta M_n|^{2p}|{\cal F}_n]\leq \Sigma_p (1+(f(\tn))^{p} \qquad \text{a.s.},
$$
and Proposition \ref{prop:sgdsc} yields the $L^{p}$-$\{\sqrt{\gamma_n}\}$ consistency rate of $(\tn)_{n \geq 1}$. 
We stress the fact that the recent contribution of \cite{BachMoulines} proves a  non-asymptotic $O(1/n)$ rate of convergence for an averaging procedure that uses  a constant step-size stochastic gradient descent. Unfortunately, \cite{BachMoulines} do not obtain a sharp constant in front of the rate $n^{-1}$ regarding the Cramer-Rao lower bound. Hence, Corollary \ref{cor:sca} yields  a stronger (and optimal) result in that case.\\

\paragraph{Assumptions $(\mathbf{H_\phi})$ and $\mathbf{(H^{\phi}_{\Sigma_p})}$ hold for many stochastic minimization problems}

We end this section by pointing out that Assumption $(\mathbf{H_\phi})$ and $\mathbf{(H^{\phi}_{\Sigma_p})}$ capture many interesting situations where the $f$ is not strongly convex and may  even not be convex in some cases. \\

$\bullet$ Before providing explicit examples, a general argument relies on the statement of Theorem 2 of \cite{Bolte}: every coercive convex continuous function $f$, which is proper and semi-algebraic (see \cite{Bolte} for some precise definitions), satisfies the KL inequality. Note that such a result holds in non-smooth situations, as stated in \cite{Bolte3}, when using sub-differential instead of gradients, but our work does not deal with non smooth-functions $f$.\\

$\bullet$ 
The \textbf{on-line logistic regression} problem deals with the minimization of $f$ defined by:
\begin{equation}\label{eq:logitf}
f(\theta)  := \mathbb{E}\left[ \log \left(1+e^{-Y <X,\theta>}\right)\right]
\end{equation}
where $X$ is a $\ER^d$ random variable and $Y | X$ takes its value in $\{-1,1\}$ with:
\begin{equation}\label{eq:logity}
P[Y=1 \, \vert X=x] = \frac{1}{1+e^{-<\tcr{x},\theta^\star>}}.
\end{equation}
We then observe a sequence of i.i.d. replications $(X_i,Y_i)$ and the baseline stochastic gradient descent sequence $(\tn)_{n \geq 1}$ is defined by:
\begin{equation}\label{eq:defonline}
\tnp = \tn + \gamma_{n+1} \frac{Y_n X_n}{1+e^{ Y_n < \tn, X_n>}} = \tn - \gamma_{n+1} \nabla f(\tn) + \gamma_{n+1} \Delta M_{n+1}.
\end{equation}

In the recent contribution \cite{BachLogit}, the author derive\tcr{s} an adaptive and non-asymptotic $O(1/n)$ rate of convergence of the averaging procedure with the help of self-concordant functions (see \cite{Nesterov1} for a basic introduction). However, the result of \cite{BachLogit} does not lead to a tight constant regarding the Cramer-Rao lower bound since the result of \cite{BachLogit} is deteriorated exponentially fast with $R$,  which corresponds to the essential supremum of the design of $|X|$. We state  the following result below.

\begin{prop}
Assume that the law of the design $X$ is compactly supported in $B_{\R^d}(0,R)$ for a given $R>0$ and is elliptic:  for any $e \in \mathcal{S}^{d-1}(\R^d)$, $Var(<X,e>) \geq 0$. Assume that $Y$ satisfies the logistic Equation \eqref{eq:logity}. Then
\begin{itemize}
\item[$i)$] $f$ defined in Equation \eqref{eq:logitf} is convex with $D^2 f$ bounded and Lipschitz. Moreover $D^2 f(\theta^\star)$ is invertible and satisfies $\mathbf{(H_{KL}^r)}$ with $r=0$.
\item[$ii)$] Recall that $\Sigma^\star$ is defined in \eqref{def:Sigma2}, the averaged sequence $(\htn)_{n \geq 1}$ built from the sequence $(\tn)_{n \geq 1}$ introduced in \eqref{eq:defonline} satisfies:
$$
\exists \, C >0 \quad \forall n \geq 1 \qquad \mathbb{E} |\hat{\theta}_n - \theta^\star|^2 \leq \frac{Tr(\Sigma^\star)}{n} + C n^{-5/4}.
$$
\end{itemize}
\end{prop}

\begin{proof}
We study $i)$.
Some straightforward computations yield $\forall \theta \in \R^d$:
$$
 \nabla f(\theta) = \mathbb{E} \left[ \frac{X \left[ e^{<X,\theta>}-e^{<X,\theta^\star>}\right]}{ \left[ 1+e^{<X,\theta>}\right]\left[ 1+e^{<X,\theta^\star>}\right]}\right] \quad \text{and} \quad D^{2} f(\theta)_{k,l} = \mathbb{E} \left[ \frac{X_k X_l e^{<X,\theta>}}{(1+e^{<X,\theta>})^2}\right]
$$
We can deduce that $\nabla f(\theta^\star)=0$ and that (see \cite{BachLogit} for example) 
$f$ is convex with
 $$<\theta-\theta^\star, \nabla f(\theta) >  = \mathbb{E} \left[ \frac{[<X,\theta>- <X,\theta^\star>] \left[ e^{<X,\theta>}-e^{<X,\theta^\star>} \right]}{\left[ 1+e^{<X,\theta^\star>}\right] \left[ 1+e^{<X,\theta>}\right]}\right] \geq 0,$$ because $(x-y)[e^x-e^y] > 0$ for every pair $(x,y)$ such that $x\neq y$.
It implies that $\theta^\star$ is the unique minimizer of $f$. Moreover, $D^2 f(\theta^\star) = \mathbb{E}\left[X X^T \frac{e^{<X,\theta^\star>}}{(1+e^{<X,\theta^\star>}}\right]$
 is invertible as soon as the design matrix is invertible. This property easily follows from the ellipticity condition on the distribution of the design:
$$
\forall e \in \mathcal{S}^{d-1}(\mathbb{R}^d) \qquad 
Var(<X,e>) = e^T \mathbb{E}[X X^T] e > 0,
$$
which proves that the Hessian $D^2f(\theta^\star)$ is invertible.

Regarding now the asymptotic norm of $|\nabla f(\theta)|$, the Lebesgue Theorem yields, $\forall e \in \mathcal{S}^{d-1}(\mathbb{R}^d) $:
{\small 
\begin{eqnarray*}
\lim_{t \longrightarrow + \infty} |\nabla f(t e)|  & = & \left| \mathbb{E}\left[ \frac{X \mathbf{1}_{< X, e> \geq 0} -  X e^{<X,\theta^\star>} \mathbf{1}_{< X, e> < 0}}{1+e^{<X,\theta^\star>}} \right] 
\right| \\
& = & \left| \left\langle \mathbb{E}\left[ \frac{X \mathbf{1}_{< X, e> \geq 0} -  X e^{<X,\theta^\star>} \mathbf{1}_{< X, e> < 0}}{1+e^{<X,\theta^\star>}} \right] ,e \right\rangle \right|\\
& \geq & \left|  \mathbb{E}\left[ \frac{<X,e> \mathbf{1}_{< X, e> \geq 0} -  <X,e> e^{<X,\theta^\star>} \mathbf{1}_{< X, e> < 0}}{1+e^{<X,\theta^\star>}} \right] \right|\\
&  \geq & \left|  \mathbb{E}\left[ \frac{<X,e> \mathbf{1}_{< X, e> \geq 0}}{1+e^{<X,\theta^\star>}} \right] \right|
\wedge \left|  \mathbb{E}\left[ \frac{<X,-e> e^{<X,\theta^\star>}\mathbf{1}_{< X,- e> \geq 0}}{1+e^{<X,\theta^\star>}} \right] \right|
\\
\end{eqnarray*}}
where we used the orthogonal decomposition on $e$ and $e^\perp$.
It then proves that for any $e \in  \mathcal{S}^{d-1}(\mathbb{R}^d)$, $\displaystyle\lim_{t \longrightarrow + \infty} |\nabla f(t e)| >0$ and a compactness and continuity argument leads to:
$$
\liminf_{|\theta|\longrightarrow + \infty} |\nabla f(\theta)| \geq \frac{ \inf_{e \in  \mathcal{S}^{d-1}(\mathbb{R}^d)} 
 \mathbb{E}\left[  <X,e>_+ \right]}{e^{R |\theta^\star|} (1+e^{R |\theta^\star|})}>0,
$$
since we assumed the design to be elliptic: $Var(<X,e>) >0$ for any unit vector $e$.
At the same time, it is also straightforward to check that:
$$\limsup_{|\theta|\longrightarrow + \infty} |\nabla f(\theta)| \leq + \infty,$$ which concludes the proof of $i)$.

We now prove $ii)$ and apply Corollary \ref{cor:ratehphi}. In that case, Assumption $\mathbf{(H_{KL}^r)}$  holds  with $r=0$.
Regarding Assumption $\mathbf{(H^{\phi}_{\Sigma_p})}$, we can observe that the martingale increments are \textit{bounded} (see \cite{BachLogit}, for example) and Inequality \eqref{eq:noisephi} is satisfied. Hence, Corollary \ref{cor:ratehphi} implies that $(\tn)_{n \geq 1}$ is a $L^{p}$-$\{\sqrt{\gamma_n}\}$ consistent sequence for any $p \geq 2$. We can therefore apply Theorem \ref{thm:avgd} for the averaging procedure $(\htn)_{n\geq 1}$, with $\Sigma^\star$ given in \eqref{def:Sigma2}. This ends the proof.
\end{proof}

$\bullet$  The \textbf{recursive quantile estimation} problem is a standard example that may be stated as follows (see, \textit{e.g.}, \cite{Duflo} for details). For a given cumulative distribution function $G$ defined over $\ER$, the problem is to find the quantile $q_{\alpha}$ such that
$G(q_{\alpha})=1-\alpha$. We assume that we observe a sequence of i.i.d. realizations $(X_i)_{i \geq 1}$ distributed with a cumulative distribution $G$. 
The recursive quantile algorithm is then defined by:

$$
\tnp= \tn - \gamma_{n+1} \left[ \mathbf{1}_{X_n \leq \tn} - (1-\alpha)\right] = 
\tn-\gamma_{n+1} [G(\tn)-(1-\alpha)] + \gamma_{n+1} \Delta M_{n+1},
$$
In that situation, the function $f\tcr{'}$ is defined by:
$$
f'(\theta)=\int_{q_{\alpha}}^{\theta} p(s) ds = G(\theta)-G(q_\alpha),
$$
where $p$ is the density with respect to the Lebesgue measure such that $G(q)=\displaystyle\int_{-\infty}^q p$. Assuming without loss of generality that $q_\alpha=0$ so that $f(0)=0$, the function $f$ is then defined by:
$$
f(\theta) := \int_{0}^\theta \int_{0}^u p(s) ds du,
$$
whose minimum is attained at $0$. It can be immediately be checked that $f''(0) \neq 0$ as soon as $p(q_{\alpha})>0$ and 
$f'(\theta) \longrightarrow 1-\alpha$ when $\theta \longrightarrow  + \infty$ while $f'(\theta) \longrightarrow - \alpha$  when $\theta \longrightarrow  - \infty$. Therefore, $f$ satisfies $(\mathbf{H_\phi})$ since $(\mathbf{H}_{KL}^r)$  and Equation \eqref{eq:KL} hold with $r=0$ and $\phi(t)=\sqrt{1+t^2}$. Again, regarding Assumption $\mathbf{(H^{\phi}_{\Sigma_p})}$, we can observe that the martingale increments are \textit{bounded} (see \cite{Cenac,Duflo}, for example). Therefore, Inequality \eqref{eq:noisephi} is obviously satisfied since $\phi$ is a monotone increasing function. We can apply Corollary \ref{cor:ratehphi} and conclude that the averaging sequence $(\htn)_{n \geq 1}$ satisfies the non-asymptotic optimal inequality: a constant $C>0$ exists such that:
$$
\forall n \geq 1 \qquad 
\mathbb{E} |\htn - q_{\alpha}|^2 \leq  \frac{\alpha (1-\alpha)}{ p(q_{\alpha}) \, n }  + C n^{-5/4}
$$

$\bullet$  The \textbf{on-line geometric median estimation}
We end this section with considerations on a problem close to the former one in larger dimensional spaces. The median estimation problem described in \cite{Cenac,Godichon} relies on the minimization of:
$$\forall \theta \in \R^d \qquad
f(\theta)  = \mathbb{E}[|X-\theta|],
$$
where $X$ is a random variable distributed over $\R^d$. Of course, our framework does not apply to this situation since 
$f$ is not $\mathcal{C}^2(\R^d,\R)$. Nevertheless, if we assume for the sake of simplicity that the support of $X$ is bounded (which is not assumed in the initial works of \cite{Cenac,Godichon}), then following the arguments of \cite{Kemperman}, the median is uniquely defined as soon as the distribution of $X$ is not concentrated on a single straight line, meaning that the variance of $X$ is elliptic in any direction of the sphere of $\R^d$. Moreover, it can be easily seen that:
$$
\lim_{|\theta| \longrightarrow + \infty} |\nabla f(\theta)| = 1,
$$
so that Equation \eqref{eq:KL}  holds with $r=0$. To apply Corollary \ref{cor:ratehphi}, it would be necessary to extend our work to this \textit{non-smooth} situation, which is beyond the scope of this paper, but that would be an interesting future subject of investigation.\\

\subsection{Organization of the paper}

The rest of the paper is dedicated to the proofs of the main results and the text is then organized as follows.
We first assume without loss of generality that $\theta^\star=0$ (and that $f(\theta^\star)=0$).
In Section \ref{sec:proofavgd}, we detail our spectral analysis of the behavior of $(\htn)_{n \geq 1}$ and  prove Theorem \ref{thm:avgd}. In particular, Proposition \ref{prop:mainproof} provides the main argument to derive the sharp exact first-order rate of convergence, and the results postponed below in  Section \ref{sec:proofavgd} only represent technical lemmas that are useful for the proof of Proposition \ref{prop:mainproof}.
Section \ref{sec:proofsgd} is dedicated to the proof of the $(L^{p}$,$\sqrt{\gamma_n})$-consistency under Assumption  
$(\mathbf{H_\phi})$ (proof of Theorem \ref{theo:weaksgd} $i)$). The generalization to the stronger situation of strong convexity (Proposition \ref{prop:sgdsc}) is left to the reader since it only requires slight modifications of the proof).

\section{Non asymptotic optimal averaging procedure (Theorem \ref{thm:avgd})}\label{sec:proofavgd}

The aim of this paragraph is to prove Theorem \ref{thm:avgd}. We will use a coupled relationship between $\htnp$ and $(\htn,\tnp)$. For this purpose, we introduce the notation for the drift at time $n$:
\begin{equation}\label{eq:drift}
\Lambda_n := \int_{0}^1 D^2 f(t \theta_n) dt\quad \textnormal{so that} \quad \Lambda_n\theta_n=\nabla f(\theta_n)
\end{equation}
 using the Taylor formula and the fact that $\theta^\star=\nabla f(\theta^\star)=0.$
The coupled evolution $(\tn,\htn) \rightarrow (\tnp,\htnp)$ is then described by the next proposition.
\begin{prop}
If we now introduce  $Z_n = (\tn,\htn)$, then we have the $2d$-dimensional recursion formula:
\begin{equation}\label{eq:evol_matriciel}
\Znp = \left( \begin{matrix}
I_d -\gamma_{n+1} \Lambda_n & 0 \\ \frac{1}{n+1} (I_d  -\gamma_{n+1} \Lambda_n) & (1-\frac{1}{n+1}) I_d
\end{matrix} \right) Z_n + \gamma_{n+1} \left( \begin{matrix} 
\Delta M_{n+1} \\ \frac{\Delta M_{n+1} }{n+1}
\end{matrix} \right).
\end{equation}
\end{prop}

\begin{proof}
We begin with the simple remark:
$$
\forall n\in \mathbb{N} \qquad  \htnp =  \htn+\frac{1}{n+1} \left( \tnp-\htn\right).
$$
Now, Equation \eqref{eq:gs} yields:
$$
\forall n\in \mathbb{N} \qquad  
\left \{
\begin{array}{c @{=} l}
    \tnp  & \, \theta_n-\gamma_{n+1}  \nabla f (\tn) + \gamma_{n+1} \Delta M_{n+1} \\
    \htnp & \, \htn (1-\frac{1}{n+1})+\frac{1}{n+1} \left(  \theta_n-\gamma_{n+1}  \nabla f (\tn) + \gamma_{n+1} \Delta M_{n+1} \right).\\
\end{array}
\right.
$$

The result then follows from  \eqref{eq:drift}.
\end{proof}

The next proposition describes the linearization procedure by replacing $\Lambda_n$ with the fixed Hessian of $f$ at $\theta^\star$.

\begin{prop}
Set $\Lambda^\star = D^2 f(\theta^\star)$ and assume that $\Lambda^\star$ is a  \tcr{positive-definite} matrix. Then, a matrix $Q \in \mathcal{O}_d(\R)$ exists such that
$\cZn = \left( \begin{matrix} Q & 0 \\ 0 & Q \end{matrix} \right)Z_n
$ satisfies:
\begin{equation}\label{eq:recZn}
\cZnp = A_n  \cZn + \gamma_{n+1} \left( \begin{matrix} 
Q \Delta M_{n+1} \\ \frac{Q \Delta M_{n+1} }{n+1}
\end{matrix} \right) 
+ \underbrace{ \gamma_{n+1} \left( \begin{matrix} Q (\Lambda^{\star} - \Lambda_n) \tn \\
Q (\Lambda_n - \Lambda^{\star}) \frac{\tn}{n+1}
\end{matrix}
\right)}_{:=\check{\upsilon}_n},
\end{equation}
where $D^\star$ is the diagonal matrix associated with the eigenvalues of $\Lambda^\star$ and
\begin{equation}\label{eq:defAn}
A_n:=\left( \begin{matrix}
I_d -\gamma_{n+1} D^\star & 0 \\ \frac{1}{n+1} (I_d  -\gamma_{n+1} D^\star) & (1-\frac{1}{n+1}) I_d
\end{matrix} \right).
\end{equation}
\end{prop}

\begin{proof}

We write $\Lambda_n = \underbrace{D^2f(\theta^{\star})}_{:=\Lambda^\star} + (\Lambda_n -  D^2f(\theta^{\star}))$ and use the eigenvalue decomposition of $\Lambda^\star$. 
\begin{equation}\label{eq:updatezn}
\Znp = \left( \begin{matrix}
I_d -\gamma_{n+1} \Lambda^\star & 0 \\ \frac{1}{n+1} (I_d  -\gamma_{n+1} \Lambda^\star) & (1-\frac{1}{n+1}) I_d
\end{matrix} \right) Z_n + \gamma_{n+1} \left( \begin{matrix} 
\Delta M_{n+1} \\ \frac{\Delta M_{n+1} }{n+1}
\end{matrix} \right) + \upsilon_n,
\end{equation}
where the linearization term $\upsilon_n$ will be shown to be negligible and is defined by
\begin{equation*}
 \upsilon_n := 
\gamma_{n+1} \left( \begin{matrix} (\Lambda^{\star} - \Lambda_n) \tn \\
(\Lambda_n - \Lambda^{\star}) \frac{\tn}{n+1}
\end{matrix}
\right).
\end{equation*}
The matrix $\Lambda^\star$  is the Hessian of $f$  at $\theta^{\star}$  and is a symmetric positive matrix, which may be reduced into a diagonal matrix $D^\star =  Diag(\mu_1^{\star}, \ldots, \mu_d^{\star})$ with positive eigenvalues in an orthonormal basis:
\begin{equation}\label{eq:defQ}
\exists \,  Q \in \mathcal{O}_d(\R) \qquad \Lambda^\star = Q^T D^\star Q \qquad \text{with} \qquad Q^T = Q^{-1}.
\end{equation}
It is natural to introduce the new sequence adapted to the spectral decomposition of $\Lambda^{\star}$ given by Equation \eqref{eq:defQ}:
\begin{equation}\label{eq:defczn}
\cZn = \left( \begin{matrix} Q & 0 \\ 0 & Q \end{matrix} \right)Z_n =  \left(\begin{matrix} Q \tn \\ Q\htn \end{matrix} \right).
\end{equation}
Using $Q \Lambda^\star = D^\star Q$, we obtain the equality described in Equation \eqref{eq:recZn}.
\end{proof}
\medskip
 The important fact about the evolution of $(\cZn)_{n \geq 1})$ is the blockwise structure of $A_n$ as $d$ blocks of $2\times 2$ matrices:

{\small
   \begin{equation}\label{eq:defblock} A_n = \left( \begin{array}{c@{}c}
     \left[\begin{array}{cccc} 1-\gamma_{n+1}\mu_1^\star & 0 & \ldots & 0 \\
      0 & 1-\gamma_{n+1}\mu_2^\star & \ldots & \vdots \\
      \vdots & \ldots & \ddots & \vdots \\
      0 & \ldots & 0 & 1-\gamma_{n+1}\mu_d^\star
    \end{array}\right] & \mathbf{0_d} \\
\left[\begin{array}{cccc} \;\frac{1-\gamma_{n+1}\mu_1^\star}{n+1} \;& \;0\; & \;\ldots\; &\; 0 \;\\
      0 & \frac{1-\gamma_{n+1}\mu_2^\star}{n+1} & \ldots & \vdots \\
      \vdots & \ldots & \ddots & \vdots \\
      0 & \ldots & 0 & \frac{1-\gamma_{n+1}\mu_d^\star}{n+1}
    \end{array}\right]  & (1-\frac{1}{n+1}) {\bf I_d}
    \end{array} \right).\end{equation}
    }

In particular, we can observe that the matrices made of components  $(i,i)$ $(i, d+i)$, $(d+i,i)$ and $(d+i,d+i)$ have a similar form. In the next proposition, we focus on the related spectrum of such $2\times 2$-matrices (the proof is left to the reader).

\begin{prop}
For $\mu \in\ER$ and $n\ge1$, set
$
E_{\mu,n} := \left( \begin{matrix} 1-\gamma_{n+1} \mu & 0 \\ \frac{1-\mu \gamma_{n+1}}{n+1} & 1-\frac{1}{n+1}  \end{matrix} \right).$
$\bullet$ If $1-\mu \gamma_{n+1} (n+1)\neq0$, define $\epsilon_{\mu,n+1} $ by:
  \begin{equation}\label{eq:enmu}
\epsilon_{\mu,n+1} := \frac{1-\mu \gamma_{n+1}}{1-\mu \gamma_{n+1} (n+1)},
\end{equation}
The eigenvalues of $E_{\mu,n}$ are then given by
$$
Sp(E_{\mu,n}) = \left\{ 1 - \mu \gamma_{n+1}, 1-\frac{1}{n+1}\right\},
$$
whereas the associated eigenvectors are:
$$
u_{\mu,n} = \left(  \begin{matrix} 1 \\ \epsilon_{\mu,n+1} \end{matrix} \right) \quad \text{and}\quad  v=
\left(  \begin{matrix} 0 \\ 1 \end{matrix} \right).$$
$\bullet$ If $1-\mu \gamma_{n+1} (n+1) = 0$, $E_{\mu,n}$  is not diagonalizable in $\R$.
\end{prop}

At this stage, we point out that the eigenvectors are modified from one iteration to another in our spectral analysis of $(\htn)_{n \geq 1}$. Lemma \ref{lemma:diff_eps} (stated in Appendix \ref{appendix:A} will be useful to assert how much the eigenvectors are moving.

\begin{rem}
The spectral decomposition of $E_{\mu,n}$ will be important below.
\begin{itemize}
\item The first important remark is that $E_{\mu,n}$ \underline{is not symmetric}. The same remark holds for $A_n$ as well as shown in Equation \eqref{eq:defblock}. This generates a non-orthonormal change of basis to reduce $E_{\mu,n}$ and $A_n$ into a diagonal form, which implies some technical complications for the study of $(\cZn)_{n \geq 1}$. 
\item To a lesser extent, it is also interesting to point out that  this ``no self-adjointness" property of $A_n$
is a new example of acceleration of convergence rates with the help of non symmetric dynamical systems. This phenomenon also occurs  for the  kinetic  diffusion dynamics \cite{Villani,GadatMiclo}) and for the Nesterov accelerated gradient descent \cite{Nesterov1,CabotEnglerGadat1} even though we do not claim that such a clear common point exists between these methods.
\item The first eigenvalue of $E_{\mu,n}$  is $1-\mu \gamma_{n+1}$, and essentially acts on the component $\tn$ of the vector $Z_n$. We then expect a contraction of $\tn$ related to $\prod_{k=1}^n (1-\mu \gamma_{k+1})$ where $\mu$ is the associated eigenvalue of the Hessian of $f$ at $\theta^\star$. 
In a sense, there is nothing new for the standard stochastic gradient descent algorithm in this last observation.

\item 
Interestingly, the second eigenvalue of $E_{\mu,n}$ is $1-(n+1)^{-1}$, which is \underline{\emph{\textbf{independent}}} of the value of $\mu$. Moreover, this eigenvalue acts on the component brought by $\htn$ in the vector $Z_n$. This key observation will be at the core of the argument for a non-asymptotic study of the Ruppert-Polyak algorithm and an important fact to obtain the adaptivity property for the unknown value of $D^\star$.
In the following section, we  obtain some helpful properties on the averaging procedure due to a careful inspection of the evolution of the eigenvalues of $E_{\mu,n}$ from $n$ to $n+1$.
\end{itemize}
\end{rem}
\noindent
The reduction of $E_{\mu,n}$ may be written as:
$$
E_{\mu,n} = \left( \begin{matrix}
1 & 0 \\ \epsilon_{\mu,n+1} & 1
\end{matrix} \right) \left( \begin{matrix}
1-\mu \gamma_{n+1} & 0 \\ 0 & 1-\frac{1}{n+1}
\end{matrix} \right) \left( \begin{matrix}
1 & 0 \\ -\epsilon_{\mu,n+1} & 1
\end{matrix} \right).
$$
Therefore, if we define the diagonal matrix $\mathcal{E}_{n,D^{\star}}$ by:
\begin{equation}\label{eq:end}
\mathcal{E}_{n,D^{\star}} = Diag(\epsilon_{\mu_1^\star,n+1}, \ldots, \epsilon_{\mu_d^\star,n+1}),\end{equation}  we then deduce the spectral decomposition of $A_n$:
\begin{equation}\label{eq:decAn}
A_n = \left( \begin{matrix}
I_d & 0 \\
\mathcal{E}_{n,D^{\star}} &  I_d
\end{matrix}
\right) 
\left( \begin{matrix}
I_d -\gamma_{n+1} D^\star & 0\\
0 & (1-\frac{1}{n+1}) I_d\\
\end{matrix}
\right) \left( \begin{matrix}
I_d & 0 \\
-\mathcal{E}_{n,D^{\star}} &  I_d
\end{matrix}
\right).
\end{equation}
We introduce the last change of basis as:
\begin{equation}\label{eq:deftzn}
\tZn := \left( \begin{matrix}
I_d & 0 \\
-\mathcal{E}_{n,D^{\star}} &  I_d
\end{matrix}
\right) \cZn.
\end{equation}
We will establish the following proposition.
\begin{prop}\label{prop:mainproof}  Assume that $\Lambda^\star$ is a \tcr{positive-definite} matrix. If $(\tn)_{n \geq 1}$ is a $(L^{p}$,$\sqrt{\gamma_n})$-consistent sequence with $p\ge 4$ and if   $\mathbf{(H_{S})}$ holds
 then the sequence $(\tZn)_{n \geq 0}= (\tZn^{(1)},\tZn^{(2)})_{n \geq 0}$ satisfies:
\begin{itemize}
\item $i)$ Some constants $(c_p)_{p \geq 1}$ exists such that:
$$
\forall n \geq 1 \qquad 
 \E \left| \tZn^{(1)}\right|^{p} \lesssim c_p \{\gamma_n\}^\frac{p}{2}.
$$
\item $ii)$ A constant $c_2$ exists such that:
$$
\forall n \geq 1 \qquad 
 \E \left| \tZn^{(2)}\right|^2 \leq \frac{{\rm Tr}(\Sigma^\star)}{n}+\frac{c_2}{n^{r_\beta}},
 $$
 where $r_\beta=\left\{ (\beta+1/2) \wedge \left(2-\beta\right)\right\}>1$ as soon as $\beta\in (1/2,1)$.
\end{itemize}
\end{prop}
Since we aim to obtain the highest possible value for the second order term $r_\beta$, we are driven to the ``optimal" choice $\beta=3/4$, which in turns implies that
$$
\forall n \in \mathbb{N}^\star \qquad \mathbb{E} |\tZn|_2^2 \leq \frac{{\rm Tr}( \Sigma^\star)}{n} + C n^{-5/4}.
$$

\begin{proof} 

\underline{Proof of $i)$:}
We first observe that the sequence in $\R^d \times \R^d$ may be written as $\tZn=(\tZn^{(1)},\tZn^{(2)})$ and Equations \eqref{eq:defczn} and \eqref{eq:deftzn} prove that $\tZn^{(1)} = Q\tn$.  Then, the  $(L^{p},\sqrt{\gamma_n})$-consistency of $(\tZn^{(1)})_{n\ge1}$ is a direct consequence of the one of $(\tn)_{n \geq 1}$.\smallskip


\noindent \underline{Proof of $ii)$:} We pick $n_0$ such that $\forall n \geq n_0: \epsilon_{\mu,n} < 0$ for any $\mu \in Sp(\Lambda^\star)$.

\textbf{Step 1: Recursion  formula}

We  first establish a recursion between $\tZn$ and $\tZnp$ that will be used in Lemma \ref{lem:recursions}. It will provide a key relationship
on the covariance between $\tZn^{(1)}$ and $\tZn^{(2)}$ and on the variance of $\tZn^{(2)}$.\smallskip

Definitions \eqref{eq:defczn}, \eqref{eq:deftzn}, the recursive link \eqref{eq:updatezn}  and the definition of $\check{\upsilon}_n$ given in Equation \eqref{eq:recZn} yield:
{\small
\begin{align*}
\tZnp  &=  \left( \begin{matrix}
I_d & 0 \\
-\mathcal{E}_{n+1,D^{\star}} &  I_d
\end{matrix}
\right) \cZnp 
\nonumber\\
&=  \left( \begin{matrix}
I_d & 0 \\
-\mathcal{E}_{n+1,D^{\star}} &  I_d
\end{matrix}
\right) \left( A_n \cZn +  \gamma_{n+1} \left( \begin{matrix} 
Q \Delta M_{n+1} \\ \frac{Q \Delta M_{n+1} }{n+1}
\end{matrix} \right) + \check{\upsilon}_n \right)
\nonumber \\
 &=  
 \left( \begin{matrix}
I_d & 0 \\
-\mathcal{E}_{n+1,D^{\star}} &  I_d
\end{matrix}
\right)  \left( \begin{matrix}
I_d & 0 \\
\mathcal{E}_{n,D^{\star}} &  I_d
\end{matrix}
\right) \left( \begin{matrix}
I_d -\gamma_{n+1} D^\star & 0\\
0 & (1-\frac{1}{n+1}) I_d\\
\end{matrix}
\right) \tZn
\nonumber\\
&+   
\gamma_{n+1} \left[\left(\begin{matrix}
Q \Delta M_{n+1} \\ (-\mathcal{E}_{n+1,D^{\star}} + \frac{I_d}{n+1}) Q \Delta M_{n+1}
\end{matrix} \right)  + \left(
\begin{matrix} Q(\Lambda^\star-\Lambda_n) \tn\\
(\mathcal{E}_{n+1,D^{\star}} - \frac{I_d}{n+1})
Q (\Lambda^\star-\Lambda_n) \tn
\end{matrix} \right)\right],
\end{align*}}
where in the third line we used the spectral decomposition of $A_n$ given by \eqref{eq:decAn}.
Since $D^2f$ is Lipschitz continuous, $\|\Lambda^\star-\Lambda_n\|=O(|\theta_n|)$. Then, we deduce that: 
{\small
\begin{align}\label{tznpeq}
\begin{cases}
\tZnp^{(1)}=(I_d -\gamma_{n+1} D^\star) \tZn^{(1)}+\gamma_{n+1} \left(Q\Delta M_{n+1}+O\left(|\theta_n|^2\right)\right)\\
\tZnp^{(2)}=(1-\frac{1}{n+1})\tZn^{(2)}+\Omega_n \tZn^{(1)}
+ \gamma_{n+1} \Upsilon_n\left(  Q\Delta M_{n+1}+O\left( |\theta_n|^2\right)\right),
\end{cases}
\end{align}}
with $$\Omega_n=(\mathcal{E}_{n,D^{\star}}-\mathcal{E}_{n+1,D^{\star}}) (I_d-\gamma_{n+1} D^{\star}) \quad \text{and}\quad\Upsilon_n= \mathcal{E}_{n+1,D^{\star}} - \frac{I_d}{n+1}.$$
%

\textbf{Step 2: $\mathbf{\ES[|\tilde{Z}_n^{(2)}|^2]=O(n^{-1})}$}
The study of $\ES[|\theta_n|^2 \tilde{Z}_n^{(2)}]$ is rather intricated as pointed in Lemma \ref{lem:recursions}. We introduce the covariance:

\begin{equation}\label{eq:correlationi}
\forall i \in \{1,\ldots,d\} \qquad \omega_n(i)=\ES[(\tZn)_i(\tZn)_{d+i}]=\ES[(\tZn^{(1)})_i(\tZn^{(2)})_i],
\end{equation}
and the useful coefficient:
\begin{equation}\label{eq:alphani}
\forall i \in \{1,\ldots,d\} \qquad \alpha_n^i= 2 \left( 1-\frac{1}{n+1}\right) \{\Omega_n\}_{i,i}.
\end{equation}

We can use the Young inequality  $ab\le \frac{\epsilon}{2} a^2+\frac{1}{2\epsilon} b^2$  with some well-chosen
$\epsilon$. 
More precisely, setting $\epsilon= n^{r}$, we obtain:
$$\ES[|\theta_n|^2 |\tilde{Z}_n^{(2)}|]\lesssim n^{r}\ES[|\theta_n|^4]+n^{-r}\ES[|\tilde{Z}_n^{(2)}|^2]\le n^{r-2\beta}+ n^{-r}\ES[|\tilde{Z}_n^{(2)}|^2].$$
Since $2\beta>1$, we know that  a $\delta>0$ exists such that $r = 2 \beta-1-\delta>0$ and
$$\frac{\ES[|\theta_n|^2 |\tilde{Z}_n^{(2)}|]}{n}\le n^{-2-\delta}+n^{\tcb{-2 \beta+\delta}}\ES|\tilde{Z}_n^{(2)}|^2.$$
Second, from Lemma \ref{lem:recursions}, for every $i\in\{1,\ldots,d\}: |\alpha_n^i| \lesssim \{n^2 \gamma_n\}^{-1}$ and 
\begin{align*}
|\alpha_n^i\omega_n(i)|&  \lesssim \frac{1}{\gamma_n n^2}\left(n^{-\delta'}\ES|\tZn^{(1)}|^2+n^{\tcb{\delta'}}\ES|\tilde{Z}_n^{(2)}|^2\right)\\
&\le n^{-2-\delta'}+ n^{\beta+\delta'-2}\ES|\tilde{Z}_n^{(2)}|^2.
\end{align*}
Plugging the two previous controls into the second statement of Lemma \ref{lem:recursions}, we get a positive $\delta$ such that a $n_0$ exists such that 
for all $n\ge n_0$:
\begin{eqnarray*}
\ES[|\tZnp^{(2)}|^2]&\le & \left( \left(1-\frac{1}{n+1}\right)^2+C [n^{\tcb{-2 \beta+\delta}}+
n^{\beta+\delta'-2}]\right)\ES[|\tZn^{(2)}|^2]+\frac{{\rm Tr}(\Sigma^\star)}{(n+1)^2}\\
& & + C \left(  n^{-(2+\delta)} +n^{-(2+\delta')} + n^{-(2+\beta/2)} + n^{-3+\beta}\right).
\end{eqnarray*}
We choose $\delta = \beta-1/2>0$ and $\delta'=1/2-\beta/2>0$. In the meantime, we also have $2+\delta \wedge 2+\delta' \wedge 2+\beta/2 \wedge 3-\beta > 2$.
According to this choice, we can apply Lemma \ref{lem:tecnique-clef} and deduce that
a $\eta>0$ exists such that:
$$ \forall n \geq 1 \qquad 
\ES[|\tZnp^{(2)}|^2]\le \frac{{\rm Tr}(\Sigma^\star)}{n+1}(1+O(n^{-\eta}).$$

\textbf{Step 3: Control of the covariance}
Owing to the previous control of $\ES[|\tilde{Z}_n^{(2)}|^2]$, one can deduce from Cauchy-Schwarz inequality that:
\begin{equation}\label{previouscontrol}
\ES[|\theta_n|^2 |\tilde{Z}_n^{(2)}|]\le \sqrt{ \ES[|\theta_n|^4]} \sqrt{ \ES[ |\tilde{Z}_n^{(2)}|^2]} \lesssim \frac{\gamma_n}{\sqrt{n}}.
\end{equation}
Plugging this control into Lemma \ref{lem:recursions} $i)$, we obtain that for all $i\in\{1,\ldots,d\}$:
$$\omega_{n+1}(i)=\left(1-\gamma_{n+1}\mu_i^\star\right)\frac{n}{n+1} \omega_n(i)+O\left(\frac{\gamma_{n+1}}{n+1}\right)+O\left(\frac{\gamma_{n+1}^2}{\sqrt{n}}\right).$$
Now, remark that  $\gamma_n\lesssim \sqrt{n}$ so that we can conclude that $\ES[|\theta_n|^2 |\tilde{Z}_n^{(2)}|]$ shall be neglected in the evolution of $(\omega_n(i))_{n \geq 1}$:
$$\omega_{n+1}(i)=\left(1-\gamma_{n+1}\mu_i^\star\right)\frac{n}{n+1} \omega_n(i)+O\left(\frac{\gamma_{n+1}}{n+1}\right).$$
From Lemma \ref{lem:contechnique2} stated in Appendix \ref{appendix:A}, we conclude that:
\begin{equation}\label{dlomega}
\forall i\in\{1,\ldots,d\} \qquad \omega_n(i)=O\left(\frac{1}{n}\right).
\end{equation}

\textbf{Step 4: Expansion of the quadratic error}
We can conclude the proof of Proposition \ref{prop:mainproof} $ii)$. From the previous upper bounds \eqref{dlomega} and \eqref{previouscontrol}, we have:
$$
\sum_{i=1}^d \alpha_n^i \omega_n(i) = O\left(\frac{1}{n^2  \gamma_n}\right) \times O\left(\frac{1}{n}\right) \quad \text{and} \quad \frac{ \ES[|\theta_n|^2 |\tilde{Z}_n^{(2)}|] }{n} = O\left( \frac{\gamma_n}{n \sqrt{n}} \right).
$$
We use these bounds  in the statement of Lemma \ref{lem:recursions} $ii)$ and deduce that:
\begin{eqnarray*}
\ES[|\tZnp^{(2)}|^2] &\le & \left(1-\frac{1}{n+1}\right)^2\ES[|\tZn^{(2)}|^2]+O \left( \frac{1}{n^3 \gamma_n}\right) +O \left( \frac{\gamma_n}{n^{\frac{3}{2}}}\right)+ O\left(\frac{\sqrt{\gamma_n}}{n^{2}}\right) \\
& \le & \left(1-\frac{1}{n+1}\right)^2\ES[|\tZn^{(2)}|^2]+O \left(\frac{1}{n^{(\frac{3}{2}+\beta)\wedge(3-\beta)}} \right)
\end{eqnarray*}
where we used that   $\gamma_n=\gamma n^{-\beta}$ so that $\sqrt{\gamma_n}n^{-2} = o(\gamma_n n^{-3/2})$ regardless the value of $\beta \in (1/2,1)$.
Applying again Lemma \ref{lem:tecnique-clef} with $r=+\infty$ and $q_{\beta}=(\frac{3}{2}+\beta)\wedge(3-\beta)$, one obtains
the announced result.
\end{proof}

\begin{rem}[About the linear case]\label{rem:linearcase} When $x\mapsto D^2 f(x)$ is constant (or also when the function $f$ to minimize is ${\cal C}^3$ with a third partial derivatives Lipschitz and null at $\theta^\star)$, we can remark that 
$\Lambda_n=\Lambda^\star$ (or that $\Lambda_n-\Lambda^\star=O(|\theta_n|^2)$). Following carefully the proof  of Lemma \ref{lem:recursions}, we can deduce that the error term $n^{-1}O(\ES[|\theta_n|^2 |\tilde{Z}_n^{(2)}|])$
vanishes (or is replaced by $n^{-1}O(\ES[|\theta_n|^3 |\tilde{Z}_n^{(2)}|])\lesssim (n^{-1}\gamma_n)^{\frac{3}{2}}$ if the $(L^6,\sqrt{\gamma_n})$-consistency holds). Hence:
\begin{equation*}
\ES[|\tZnp^{(2)}|^2]\le \left(1-\frac{1}{n+1}\right)^2\ES[|\tZn^{(2)}|^2]+O (n^{-3}\gamma_n^{-1})+O\left(\frac{\sqrt{\gamma_n}}{n^2}\right).
\end{equation*}
The rate is then optimized by choosing $\beta=2/3$, leading to an exponent $n^{-\frac{4}{3}}$.
\end{rem}

\section{$(L^{p}$,$\sqrt{\gamma_n})$-consistency - (Theorem \ref{theo:weaksgd})}\label{sec:proofsgd}
The  main objective of this section is to prove Theorem \ref{theo:weaksgd} $iii)$.  Our analysis is based on a Lyapunov-type approach with the help of  $V_p:\ER^d\rightarrow\ER$ 
defined for a given $p\ge1$ by:
$$V_p(x)=f^p(x)\exp(\phi(f(x)).$$
We have the following result:
\begin{thm}[Convergence rate of $(\tn)_{n \geq 1}$ with $\mathbf{H_{\phi}}$ \label{thm:sgd}]
Let $p\ge1$ and  assume  $(\mathbf{H_\phi})$ and $\mathbf{(H^{\phi}_{\Sigma_p})}$. Let $(\gamma_n)_{n\rightarrow}$ be a non-increasing sequence such that $\gamma_n\rightarrow0$ as $n\rightarrow+\infty$. Then,
\begin{itemize}
\item[$i)$] An integer $n_0\in\mathbb{N}$ \tcr{and some positive  $c_1$ and $c_2$} exist such that 
\begin{equation}\label{eq:recursionvvp}
\forall n\ge n_0,\quad\ES[V_p(\theta_{n+1})]\le (1-c_1\gamma_{n+1}) \ES[V_p(\theta_n)]+c_2\gamma_{n+1}^{p+1}.
\end{equation}
\item[$ii)$] Furthermore, if $\gamma_n-\gamma_{n+1}=o(\gamma_{n+1}^2)$ as $n\rightarrow+\infty$, then 

$$
\forall n \geq 1 \qquad \mathbb{E} \left[ V_p(\theta_n) \right] \leq C_p \{\gamma_n\}^{p}.
$$
 In particular, 
 $$\forall n \geq 1 \qquad \mathbb{E} [f^p(\tn)]\le C_p \{\gamma_n\}^{p}.$$
 \end{itemize}
\end{thm}
Note that the condition  $\gamma_n-\gamma_{n+1}=o(\gamma_{n+1}^2)$ is satisfied when $\gamma_n=\gamma n^{-\beta}$ with $\beta\in(0,1)$. Therefore,  Theorem \ref{theo:weaksgd} $iii)$ holds true.
\noindent 

To prove Theorem \ref{thm:sgd} $i)$, we need some technical results related to $\phi$ and $V_p$. The first result is a simple sub-additive property on $\phi$ that essentially relies on the concavity property on $[x_0,+\infty)$.
\begin{lemma}\label{lem0WMR} Assume that $\phi$ satisfies $(\mathbf{H_\phi})(i)$, then a constant $c_\phi$ exists such that for all $x, y\in\ER_+$:
$$\phi(x+y)\le \phi(x)+\phi(y)+c_\phi.$$
\end{lemma}
\begin{proof} Since $\phi'' \leq 0$ on $[x_0,+\infty)$, the function $\phi$ is concave on $[x_0,+\infty)$. Hence, the function $x\mapsto \phi(x+y)-\phi(x)$ is decreasing on $[x_0,+\infty)$ and we deduce that:
$$\forall x \ge x_0 \quad \phi(x+y)\le \phi(x)+\phi(x_0+y)-\phi(x_0).$$
Since $\phi'$ is decreasing on $[x_0,+ \infty)$, then $\phi'$ is upper-bounded and  a constant $C>0$ exists such that $\phi(y+x_0)\le \phi(y)+C x_0$. We then deduce that:
\begin{equation}\label{eq:cas1}
\forall x \ge x_0 \quad \forall y \geq 0 \qquad 
\phi(x+y)\le \phi(x)+ \phi(y)+C x_0-\phi(x_0).
\end{equation}
In the other situation when $x\le x_0$, the fact that $\phi$ is non-decreasing yields and Equation \eqref{eq:cas1} applied at point $x_0$ yields:
$$\phi(x+y)\le \phi(x_0+y)\le \phi(y)+C x_0\le \phi(x)+ \phi(y)+C x_0.$$
We then obtain the desired inequality for any value of $x$ and $y$ in $\R_+$.
\end{proof}
The second key element of our study is a straightforward computation of the first and second derivatives of $V_p$.

\begin{lemma} \label{lem1WMR}For any $p \in \mathbb{N}^\star$ and any $x\in \R^d\setminus \{\theta^\star\}$, we have:
\begin{itemize}
\item[$i)$]
$$\nabla V_p(x)=V_p(x)\left(p \frac{\nabla f(x)}{f(x)}+\phi'(f(x))\nabla f(x)\right).$$
\item[$ii)$]
$$ D^2 V_p(x)=V_p(x)\left[\psi_1(x)\nabla f(x)\otimes \nabla f(x)+ \psi_2(x) D^2 f(x)\right],$$
where $\psi_1$ and $\psi_2$ are given by:
{\small $$\hspace{-1cm}\psi_1(x):=\left( \frac{p}{f(x)}+\phi'(f(x))\right)^2-\frac{p}{f^2(x)}+\phi''(f(x)) \quad \text{and} \quad \psi_2(x):=\frac{p}{f(x)}+\phi'(f(x)).$$}
\end{itemize}

\end{lemma}
\begin{lemma}\label{lem2WMR}
Assume that $f$ satisfies $(\mathbf{H_\phi})$, then one has
\begin{itemize}
\item[$i)$] A constant $\alpha>0$ exists such that:
$$\inf_{x\in\ER^d}\frac{\langle \nabla V_p(x),\nabla f(x)\rangle}{V_p(x)}\ge\alpha>0.$$ 
\item[$ii)$] For any matrix norm $\|\,.\,\|$, a positive constant $C>0$ exists such that for any $\xi\in\ER^d$,
$$ \|D^2 V_p(\xi)\|\le C\left( V_{p-1}(\xi)+ \frac{V_p(\xi)}{1+|\nabla f(\xi)|^2}\right).$$
\end{itemize}

\end{lemma}
\begin{proof}
Below, $C$ refers to a large enough constant independent of $\xi$ whose value may change from line to line.

\noindent
 $i)$ We apply  Lemma \ref{lem1WMR} $i)$ and obtain that:
{\small $$\forall x \in \R^d \setminus \{\theta^\star\} \qquad \frac{\langle \nabla V_p(x),\nabla f(x)\rangle}{V_p(x)}=p \frac{\|\nabla f(x)\|^2}{f(x)}+\phi'(f(x))\|\nabla f(x)\|^2.$$}
The result then follows from Assumption $(\mathbf{H_\phi}) ii)$ and a continuity argument around $\theta^\star$.

%

\noindent $ii)$ 
 We apply Lemma \ref{lem1WMR} $ii)$ and write that $\forall y \in\R^d$:
{\small 
\begin{eqnarray*}
\frac{\langle y, D^2 V_p(\xi) y\rangle}{\|y\|^2} & = & V_p(\xi) \left[ \psi_1(\xi) \langle y, \nabla f(\xi)\otimes \nabla f(\xi)y \rangle + \psi_2(\xi) \langle y, D^2f(\xi)y\rangle \right]\\
& \leq &   V_p(\xi) \left(  \left[ \frac{2 p^2}{f^2(\xi)} + 2\{\phi'(f(\xi))\}^2 - \frac{p}{f^2(\xi)}+\phi''(f(\xi))\right] \|\nabla f(\xi)\|^2 \right.\\
 & &
\left. + \left[ \frac{p}{f(\xi)}+\phi'(f(\xi)) \right] \|D^2f(\xi)\|\right).
\end{eqnarray*}}
We now apply assumption $\bf{H_\phi}$: a large enough constant $C$ exists such that:
{\small $$
\frac{\|\nabla f(\xi)\|^2}{f^2(\xi)} \leq \frac{C}{f(\xi)} \qquad \text{and} \qquad \phi'(f(\xi))^2 \|\nabla f(\xi)\|^2 \leq  C \phi'(f(\xi)).
$$}

Since $\xi \longmapsto \|D^2f(\xi)\|$ is bounded under Assumption $\bf{H_\phi}$ from the norm equivalence in any  finite dimensional  real vector space, we then have that:
{\small 
\begin{eqnarray*}
\frac{\langle y, D^2 V_p(\xi) y\rangle}{\|y\|^2} &\leq & C V_p(\xi) \left[ \frac{1}{f(\xi)}+ \phi'(f(\xi))+\phi''(f(\xi))\right]\\
& \leq& C V_{p-1}(\xi) + \frac{C V_p(\xi)}{1+\|\nabla f(\xi)\|^2} 
 \left( 1+\|\nabla f(\xi)\|^2\right) \left(\phi'(f(\xi))+\phi''(f(\xi))\right).
\end{eqnarray*}
}
\noindent Since $\phi''(u)$ is negative for $u$ large enough, that $\phi'$ is bounded (it is a  non-increasing function on $[x_0,+\infty)$) and that Assumption $\bf{H_\phi}$ implies that $\lim \phi'(f(\xi))|\nabla f(\xi)|^2<+\infty$, we then deduce that:
$$\sup_{\xi\in\ER^d}(\phi'(f(\xi))+\phi''(f(\xi))(1+|\nabla f(\xi)|^2)<+\infty.$$
Hence, 
$$
\forall y \in \R^d \qquad 
\frac{\langle y, D^2 V_p(\xi) y\rangle}{\|y\|^2} \leq C \left( V_{p-1}(\xi)+\frac{V_p(\xi)}{1+\|\nabla f(\xi)\|^2}\right).
$$
The second assertion follows.\end{proof}
%
%
%

The next  lemma  will be useful to produce an efficient descent inequality.
\begin{lemma} \label{lem3WMR}
Suppose that $\mathbf{H}_{\phi}$ holds and consider $\rho\in[0,1]$. For any  $\gamma>0$, $\varepsilon>0$ define $\xi_{\gamma,\varepsilon,x}=x+\rho\gamma \left(-\nabla f(x)+\varepsilon\right)$. Then,
\begin{itemize}
\item[$i)$] A $\gamma_0>0$, a constant $C>0$ independent of $\rho$  and $\varepsilon>0$ exist such that for any $\gamma\in[0,\gamma_0]$  such that:
$$ \qquad f(\xi_{\gamma,\varepsilon,x})\le f(x)+C\gamma |\varepsilon|^2.$$
\item[$ii)$] If $ 2 \gamma \|D^2 f\|_\infty\le 1$, then $\forall  \tcb{\rho >0} \, :  \exists \, c_\rho>0 : \forall x \in \R^d:${\small 
\begin{eqnarray*}
\lefteqn{
\gamma^2 D^2 V_p(\xi_{\gamma,\varepsilon,x}) \left(-\nabla f(x)+\varepsilon\right)^{\otimes 2}} \\
&\le&  C  (1+|\varepsilon|^{2(p+1)})\exp(\phi(\gamma |\varepsilon|^2)) \left( \tcb{\rho} \gamma V_{p}(x)+ \gamma^2 V_p(x)+ (c_{\rho}+1)\gamma^{p+1}\right).
\end{eqnarray*}}
\end{itemize}
\end{lemma}
\begin{proof} $C$ is a positive constant whose value may change from line to line.

\noindent $i)$ Using the Taylor formula, a $\tilde{\xi}$ exists on the segment $[x,\xi_{\gamma,\varepsilon,x}]$ such that:
$$ f(\xi_{\gamma,\varepsilon,x})=f(x)-\rho\gamma \|\nabla f(x)\|^2+\rho\gamma \langle \nabla f(x),\varepsilon\rangle+\frac{\rho^2\gamma^2}{2} D^2 f(\tilde{\xi}) \left(-\nabla f(x)+\varepsilon\right)^{\otimes 2}.$$
$\mathbf{H}_{\phi}$ implies that   $D^2f$ is upper bounded  and  $\|a+b\|^2 \leq 2( \|a\|^2+\|b\|^2)$ yields:
$$D^2 f(\tilde{\xi}) \left((-\nabla f(x)+\varepsilon\right)^{\otimes 2}\le 2  \|D^2 f\|_\infty  \left(\|\nabla f(x)\|^2+\|\varepsilon\|^2\right).$$
By the elementary inequality $|\langle u, v\rangle|\le \frac{1}{2}(\|u\|^2+\|v\|^2)$  we deduce that:
{\small 
\begin{eqnarray*}
f(\xi_{\gamma,\varepsilon,x}) & \le & 
f(x)-\rho\gamma \|\nabla f(x)\|^2+\rho\gamma \langle \nabla f(x),\varepsilon\rangle+ C \frac{\rho^2\gamma^2}{2}  \left( \|\nabla f(x)\|^2+\|\varepsilon\|^2\right)  \\
& \le & f(x) + \rho \gamma \left[ \frac{-1}{2}+\rho\gamma  \|D^2 f\|_\infty \right] \|\nabla f(x)\|^2 + \left[ \frac{\rho \gamma}{2}+ \|D^2 f\|_\infty   \rho^2\gamma^2  \right] \|\varepsilon\|^2  \\
& \le & f(x) +  \rho \gamma   \|\varepsilon\|^2 \leq  f(x) +   \gamma   \|\varepsilon\|^2,
\end{eqnarray*}}
where in the last line we use that $\rho \leq 1$ and the condition $\gamma \|D^2 f\|_\infty \leq 1/2$. The result follows by choosing $\gamma_0 \leq C^{-1}$.\smallskip

\noindent $ii)$ We divide the proof into 4 steps.

\noindent $\bullet$ \underline{Step 1: Comparison between $V_r(\xi_{\gamma,\varepsilon,x})$ and $V_r(x)$.} 
Let $r\ge0$. Since $\phi$ is non-decreasing, one first deduces from $i)$ that a constant $C>0$ exists such that:
$$
V_r(\xi_{\gamma,\varepsilon,x})\le (f(x)+C\gamma \|\varepsilon\|^2)^r\exp\left(\phi(f(x)+\gamma \|\varepsilon\|^2)\right).
$$
The sub-additivity property of Lemma \ref{lem0WMR}  associated with $(|a|+|b|)^r \leq 2^r (|a|^r+|b|^r)$ yields:
$$
V_r(\xi_{\gamma,\varepsilon,x})\le 2^r\left( f^r(x)+(C\gamma)^r \|\varepsilon\|^{2r}\right)e^{\phi(f(x))+\phi(\gamma \|\varepsilon\|^2)+c_\phi}.
$$
Setting $\teps=(1+\|\varepsilon\|^{2r})\exp(\phi(\gamma \|\varepsilon\|^2)$,
and using that $V_0 = e^{\phi(f)}$:
{\small \begin{eqnarray} 
\forall r \ge 0 \quad \exists \, C_r>0 \qquad 
V_r(\xi_{\gamma,\varepsilon,x})  & \le & C_r \exp(\phi(\gamma \|\varepsilon\|^2)\left[V_r(x)+\gamma^r \|\varepsilon\|^{2r} V_0(x)\right]\nonumber \\
& \leq & 
C_r \exp(\phi(\gamma \|\varepsilon\|^2)\left[(1+\|\varepsilon\|^{2r}) V_r(x)+ \gamma^r \|\varepsilon\|^{2r}\right]\nonumber\\
& \le &  C_r\teps\left[  V_r(x)+ \gamma^r  \right].\label{djslkjlds}
\end{eqnarray}}
where in the second line, we used that $V_0\le c(1+V_r)$.


\noindent $\bullet$ \underline{Step 2: Upper bound of $D^2 V_p({\xi_{\gamma,\varepsilon,x}})\| .\|\nabla f(x)\|^2$.} 
We apply Lemma \ref{lem2WMR} $ii)$ with $\xi =\xi_{\gamma,\varepsilon,x}$ and we obtain that:
{\small \begin{eqnarray*}
\|D^2 V_p({\xi_{\gamma,\varepsilon,x}})\| .\|\nabla f(x)\|^2 & \leq & C \left( V_{p-1}({\xi_{\gamma,\varepsilon,x}})+\frac{V_p({\xi_{\gamma,\varepsilon,x}})}{1+\|\nabla f({\xi_{\gamma,\varepsilon,x}})\|^2}\right)\|\nabla f(x)\|^2  \\
& \lesssim & \left( T_{\epsilon,\gamma,p-1} [V_{p-1}(x)+\gamma^{p-1}] + \frac{T_{\varepsilon,\gamma,p}[V_p(x)+\gamma^p]}{1+\|\nabla f(\xi_{\gamma,\varepsilon,x})\|^2}
\right)\|\nabla f(x)\|^2 \\
& \lesssim &   T_{\epsilon,\gamma,p-1} V_{p-1}(x) [\|\nabla f(x)\|^2 + \gamma^{p-1}\|\nabla f(x)\|^2] \\
& & + T_{\varepsilon,\gamma,p} \frac{ \|\nabla f(x)\|^2}{1+\|\nabla f(\xi_{\gamma,\varepsilon,x})\|^2}[V_p(x)+\gamma^p].\\
\end{eqnarray*}}
Under Assumption $(\mathbf{H}_{\phi})$, $V_{p-1}(x)\|\nabla f(x)\|^2\le CV_p(x)$ and $\gamma^{p-1} \|\nabla f(x)\|^2 \leq C \gamma^{p-1} f(x) \leq C \gamma^{p-1} (1+V_p(x))$.
From the boundedness of  $\gamma$ and the trivial inequality since $ T_{\epsilon,\gamma,p-1} \leq 2 T_{\epsilon,\gamma,p}$,
we then deduce that
{\small \begin{eqnarray}
\|D^2 V_p({\xi_{\gamma,\varepsilon,x}})\| .\|\nabla f(x)\|^2 & \lesssim &  T_{\epsilon,\gamma,p-1} [V_p(x) + \gamma^{p-1}]+
\frac{T_{\epsilon,\gamma,p} [V_p(x) + \gamma^{p}]  \|\nabla f(x)\|^2}{1+\|\nabla f(\xi_{\gamma,\varepsilon,x})\|^2} \nonumber \\
& \lesssim &  T_{\epsilon,\gamma,p}\left[  [V_p(x) + \gamma^{p-1}]+
  \frac{[V_p(x) + \gamma^{p}]  \|\nabla f(x)\|^2}{1+\|\nabla f(\xi_{\gamma,\varepsilon,x})\|^2} \right],
\label{controlgrad1}
\end{eqnarray}
}
and we are forced to produce an upper bound of  $\frac{\|\nabla f(x)\|^2}{1+|\nabla f(\xi_{\gamma,\varepsilon,x})|^2}.$
According to the Taylor formula, a $\xi'$ exists in $[x,\xi_{\gamma,\varepsilon,x}]$ such that:
$$\nabla f(x)=\nabla f(\xi_{\gamma,\varepsilon,x}) - \rho\gamma D^2f (\xi') \left(-\nabla f(x)+\varepsilon\right),$$
and the triangle inequality yields:
$$\|\nabla f(x)\|\le \|\nabla f(\xi_{\gamma,\varepsilon,x})\|+\| D^2 f\|_\infty \gamma(\|\nabla f(x)\|+\|\varepsilon\|),$$
so that:
$$
\|\nabla f(x)\| \le (1-\|D^2f\|_{\infty} \gamma)^{-1} \left( \|\nabla f(\xi_{\gamma,\varepsilon,x})\|+\|\varepsilon\|\right).
$$

The elementary inequality $(u+v)^2\le 2(u^2+v^2)$ leads to:
$$\|\nabla f(x)\|^2\le 8 \|\nabla f(\xi_{\gamma,\varepsilon,x})\|^2+\|\varepsilon\|^2).$$
As a consequence, for a large enough constant $C$, we have that:
$$\left(\frac{\|\nabla f(x)\|^2}{1+\|\nabla f(\xi)\|^2}+|\varepsilon|^2\right)\le C(1+\|\varepsilon\|^2).$$
Plugging this inequality in \eqref{controlgrad1} yields:
$$
\|D^2 V_p({\xi_{\gamma,\varepsilon,x}})\| .\|\nabla f(x)\|^2\lesssim
 T_{\epsilon,\gamma,p} \left(    \gamma^{p-1} +  [V_p(x)+\gamma^p] (1+\|\varepsilon\|^2)\right),
$$
and since $T_{\varepsilon,\gamma,p} (1+\|\varepsilon\|^2) \leq 3 T_{\varepsilon,\gamma,p+1}$, we then conclude that:
\begin{equation}\label{controlgrad2}
\|D^2 V_p({\xi_{\gamma,\varepsilon,x}})\| .\|\nabla f(x)\|^2\lesssim
 T_{\epsilon,\gamma,p+1} \left(    \gamma^{p-1} +  V_p(x)\right),
\end{equation}

\noindent $\bullet$ \underline{Step 3: Upper bound of $D^2 V_p({\xi_{\gamma,\varepsilon,x}})\| .\|\epsilon\|^2$.} 
We focus on the noise part $\varepsilon$. Using  \eqref{djslkjlds} and Lemma \ref{lem2WMR} $ii)$ once again, we have that:
\begin{equation}\label{controlnoise}
\|D^2 V_p({\xi_{\gamma,\varepsilon,x}})\| .\|\varepsilon\|^2\lesssim T_{\varepsilon,\gamma,p+1}\left( V_{p-1}(x)+V_p(x)+\gamma^{p-1}\right).
\end{equation}

\noindent $\bullet$ \underline{Step 4: Upper bound of $D^2 V_p({\xi_{\gamma,\varepsilon,x}})\| (-\nabla f(x)+\varepsilon)^{\otimes 2}$.} 
We use Equations \eqref{controlgrad2} and \eqref{controlnoise} and obtain:
$$\gamma^2 D^2 V_p({\xi_{\gamma,\varepsilon,x}})(-\nabla f(x)+\varepsilon)^{\otimes 2}\le  CT_{\varepsilon,\gamma,p+1}( \gamma^2 V_{p-1}(x)+\gamma^2 V_p(x)+\gamma^{p+1}).$$
To obtain the result, it is now enough to prove for any $\rho>0$, a constant $c_\rho$ exists such that: 
$$\gamma^2 V_{p-1}(x) \le \rho \gamma V^p(x)+c_{\rho} \gamma^{p+1}.$$
To derive this key comparison, we use the Young inequality
$uv\le \frac{u^{\bar{p}}}{\bar{p}}+\frac{v^{\bar{q}}}{\bar{q}}$
when $1/\bar{p}+1/\bar{q}=1$. In particular, we choose $u=\tilde{\rho}\gamma^{\frac{p-1}{p}}V^{p-1}(x)$, 
$v = \gamma^{1+1/p} \tilde{\rho}^{-1}$, $\bar{p}=p/(p-1)$, $\bar{q}=p$ and obtain that
\begin{eqnarray*}
\gamma^2 V_{p-1}(x)& =& \exp(\phi(\gamma \|\epsilon\|^2)) \gamma^2 f^{p}(x) \\
&\leq &
\exp(\phi(\gamma \|\epsilon\|^2)) \left[ \frac{p-1}{p} \left( \tilde{\rho} \gamma^{(p-1)/p} f^{p-1}(x)\right)^{p/(p-1)} + \frac{\gamma^{p+1} }{p \tilde{\rho}^p} \right]\\
& \leq & \frac{p-1}{p} \tilde{\rho}^{p/(p-1)} \gamma V_p(x) + p^{-1} \tilde{\rho}^{-p} \gamma^{p+1} \exp(\phi(\gamma \|\epsilon\|^2)) \\
& \leq & \frac{p-1}{p} \tilde{\rho}^{p/(p-1)} \gamma V_p(x) + p^{-1} \tilde{\rho}^{-p} \gamma^{p+1} V_0(x) \\
\end{eqnarray*}
Using  $V_0 \leq C (1+V_p)$ once again, we then deduce that for any $\rho>0$, a constant $c_\rho$ exists such that:
$$\gamma^2 V_{p-1}(x) \leq  \rho \gamma V_p(x)+c_{\rho} \gamma^{p+1}.
$$
We obtain the final upper bound:$ \forall \rho>0 \, ,  \exists \,  c_\rho>0 \, , \forall x \in \R^d:$
$$ \gamma^2 D^2 V_p(\xi_{\gamma,\varepsilon,x}) \left(-\nabla f(x)+\varepsilon\right)^{\otimes 2}\le  CT_{\epsilon,\gamma,p+1} \left( \rho \gamma V_{p}(x)+ \gamma^2 V_p(x)+ (c_{\rho}+1)\gamma^{p+1}\right).$$
\end{proof}

\noindent
We now focus on the proof of Theorem \ref{thm:sgd} $i)$.
\begin{proof}[Proof of Theorem \ref{thm:sgd}]

$i)$ We apply the second order Taylor formula to $V_p$ and obtain that:
\begin{eqnarray*}
V_p(\theta_{n+1})&=&V_p(\theta_n)-\gamma_{n+1}\langle \nabla V_p(\theta_n),\nabla f(\theta_n)\rangle+ \gamma_{n+1} \langle V_p(\tn),\Delta M_{n+1}\rangle \\
& &  +\frac{\gamma_{n+1}^2}{2} D^2 V_p(\xi_{n+1})(-\nabla f(\theta_n)+\Delta M_{n+1})^{\otimes 2},\end{eqnarray*}
where $\xi_{n+1}=\theta_n+\rho\Delta \theta_{n+1}$, $\rho\in[0,1]$.
Using Lemma \ref{lem2WMR} $i)$, we obtain that a $\alpha>0$ exists such that: 
\begin{equation}\label{eq:rappelfort}
\forall n \in \mathbb{N}^\star \qquad   V_p(\theta_n)-\gamma_{n+1}\langle \nabla V_p(\theta_n),\nabla f(\theta_n)\rangle \leq V_p(\tn)(1-\alpha \gamma_{n+1}).
\end{equation}
Moreover, we have that $\mathbb{E}[\gamma_{n+1} \langle V_p(\tn),\Delta M_{n+1}\rangle \, \vert \, \mathcal{F}_n] = 0$. Finally, 
Lemma  \ref{lem3WMR} $ii)$ shows that a constant $C>0$ exists such that for any $\rho>0$ , for all $n \in \mathbb{N}^\star$, $c_\rho$ exists such that:

\begin{eqnarray*}
\lefteqn{\frac{\gamma_{n+1}^2}{2} D^2 V_p(\xi_{n+1})(-\nabla f(\theta_n)+\Delta M_{n+1})^{\otimes 2}}\\
&  \leq & C T_{\Delta M_{n+1},\gamma_{n+1},p+1} 
\left( \rho \gamma_{n+1} V_{p}(\tn)+ \gamma_{n+1}^2 V_p(\tn)+ (c_{\rho}+1)\{\gamma_{n+1}\}^{p+1}\right).
\end{eqnarray*}
This last upper bound associated with \eqref{eq:rappelfort} and Assumption $\mathbf{(H_{\Sigma_p}^\phi})$ yields:
{\small 
\begin{eqnarray*}
\lefteqn{\mathbb{E}\left[V_p(\theta_{n+1}) \, \vert \, \mathcal{F}_n \right] }\\
& \leq & (1-\alpha \gamma_{n+1})  V_p(\tn) +\\
& &  C
\left( \rho \gamma_{n+1} V_{p}(\tn)+ \gamma_{n+1}^2 V_p(\tn)+ (c_{\rho}+1)\{\gamma_{n+1}\}^{p+1}\right) \mathbb{E}  \left[ T_{\Delta M_{n+1},\gamma_{n+1},p+1} \, \vert \, \mathcal{F}_n \right] \\
& \leq &  (1-\alpha \gamma_{n+1}) V_p(\tn) +  C \Sigma_p
\left( \rho \gamma_{n+1} V_{p}(\tn)+ \gamma_{n+1}^2 V_p(\tn)+ (c_{\rho}+1)\{\gamma_{n+1}\}^{p+1}\right) \\
& \leq & (1-(\alpha-\rho C \Sigma_p) \gamma_{n+1}+C \Sigma_p \gamma_{n+1}^2)  V_p(\tn) + (1+c_\rho) C \Sigma_p \{\gamma_{n+1}\}^{p+1}.
\end{eqnarray*}
}
We now choose $\rho$ such that $\rho C \Sigma_p=  \frac{\alpha}{2}$ and determine that two non-negative constants $c_1$ and $c_2$ exist such that $\forall n \in \mathbb{N}^\star $:
\begin{equation}\label{eq:recurrence}
\mathbb{E}\left[V_p(\theta_{n+1}) \, \vert \, \mathcal{F}_n \right] \leq 
 \left(1-\frac{\alpha}{2} \gamma_{n+1} + c_1 \gamma_{n+1}^2 \right) V_p(\tn) + c_2 \{\gamma_{n+1}\}^{p+1}.
\end{equation}
\tcr{ Theorem \ref{thm:sgd} $i)$ easily follows by taking the expectation and by using that $c_1\gamma_{n+1}\le \alpha/4$ for $n$ large enough}.\smallskip

\tcr{\noindent $ii)$ We prove by induction that a large enough $
C>0$ exists such that:
\begin{equation}\label{inductionss}
\forall n \in \mathbb{N}^\star \qquad \mathbb{E} \left[ V_p(\tn) \right] \leq C \left\{\gamma_{n}\right\}^{p}.
\end{equation}
Since $\gamma_n-\gamma_{n+1}=o(\gamma_{n+1}^2)$ as $n\rightarrow+\infty$
$$   \left(\frac{\gamma_n}{\gamma_{n+1}}\right)^{p}\le 1+o(\gamma_{n+1})\quad\textnormal{as $n\rightarrow+\infty$},$$
a sufficiently large $n_1$ exists such that
\begin{equation}\label{eq:tecrec}
\forall n \geq n_1 \qquad 0 \leq (1-c_1 \gamma_{n+1}) \left(\frac{\gamma_n}{\gamma_{n+1}}\right)^p \leq 1-\frac{c_1}{2} \gamma_{n+1}.\end{equation}
We can choose $C_1$ large enough such that Equation \eqref{inductionss} holds true for any $n \leq n_1$ with $C \geq C_1$.
For any $n_1\in\mathbb{N}$, the result holds for any $n\le n_1$. Assuming that the property holds at a given rank $n\ge n_1$, we then have:
\begin{eqnarray*}
\ES[V_p(\theta_{n+1})] &\le  & (1-c_1\gamma_{n+1}) \ES[V_p(\theta_n)]+c_2\{\gamma_{n+1}\}^{p+1}.\\
& \leq & (1-c_1\gamma_{n+1}) C\gamma_{n}^p+c_2\{\gamma_{n+1}\}^{p+1}.\\
& \leq & C \{\gamma_{n+1}\}^{p} \left[ \left(\frac{\gamma_n}{\gamma_{n+1}}\right)^{p} (1 - c_1 \gamma_{n+1})+ \frac{c_2}{C}  \gamma_{n+1} \right] \\
& \leq & C \{\gamma_{n+1}\}^{p} \left[   1 - \left(\frac{c_1}{2}-\frac{c_2}{C}\right) \gamma_{n+1}  \right] \\
\end{eqnarray*}
where we used Equation \eqref{eq:recursionvvp}, the induction property \eqref{inductionss} and Inequality  \eqref{eq:tecrec}. If we choose  $C \geq C_2 = \frac{c_2}{2 c_1}$, then $\mathbb{E} [V_p(\tn)] \leq C \gamma_{n}^p \Longrightarrow \mathbb{E} [V_p(\tnp)] \leq C \left\{\gamma_{n+1}\right\}^p$. This ends the proof of $ii)$.}
\end{proof}


\section*{Acknowledgments}
The authors gratefully acknowledge J\'er\^ome Bolte and Gersende Fort for stimulating discussions on the Kurdyka-\L ojasiewicz inequality and averaged stochastic optimization algorithms.
\bibliographystyle{alpha}
\bibliography{averaging_bib}
\appendix
\section{Technical lemmas for  Theorem 2}\label{appendix:A} 
The next lemma is important to obtain the stability of the change of basis from one iteration to another in our spectral analysis of $(\htn)_{n \geq 1}$.
\begin{lemma}\label{lemma:diff_eps} Assume that $\gamma_n=\gamma n^{-\beta}$ with $\beta \in (0,1)$. Let  $\mu>0$.  Then, a constant $C$ and an integer $n_0$ exist  such that
$$
\forall n \geq n_0, \qquad 
\left|\epsilon_{\mu,n}-\epsilon_{\mu,n+1} \right| \leq C n^{\beta-2}
$$
\end{lemma}

\begin{proof}
We choose $n_0$ such  that $1-\mu\gamma_{n}n<0$ for all $n\ge n_0$.
Then, the desired inequality comes from a direct computation:
\begin{eqnarray*}
\epsilon_{\mu,n}-\epsilon_{\mu,n+1} &= &\frac{1-\mu \gamma_{n}}{1-\mu \gamma_{n} n} - \frac{1-\mu \gamma_{n+1}}{1-\mu \gamma_{n+1} (n+1)} \\
& = & \frac{(1-\mu \gamma_{n})(1-\mu \gamma_{n+1} (n+1)) -(1-\mu \gamma_{n+1})(1-\mu \gamma_{n} n) }{(1-\mu \gamma_{n} n)(1-\mu \gamma_{n+1} (n+1))} \\
& = & \mu \frac{ (\gamma_{n+1}-\gamma_n) + (n  \gamma_n - (n+1) \gamma_{n+1})+\mu \gamma_n \gamma_{n+1} }{(1-\mu \gamma_{n} n)(1-\mu \gamma_{n+1} (n+1))}
\end{eqnarray*}
Now, if $C$ denotes a constant that only depends on $\mu$ and $\beta$ (whose value may change from line to line), we then have the following inequalities:
$$|\gamma_{n+1}-\gamma_n| \leq C n^{-(1+\beta)}, \; \left| n  \gamma_n - (n+1) \gamma_{n+1} \right| \leq C n^{-\beta} \; \text{and} \; 
\gamma_n \gamma_{n+1} \leq C n^{-2\beta}.
$$
 Since $\beta<1$, the denominator is equivalent to $n^{2-2\beta}$ and we obtain that
\begin{equation}\label{eq:contdifepsss}
\left| \epsilon_{\mu,n}-\epsilon_{\mu,n+1} \right| \leq C \frac{n^{-\beta}}{n^{2-2\beta}} = C n^{\beta-2},
\end{equation}
which ends the proof.
\end{proof}

\begin{lemma} \label{lem:recursions} Under the assumptions of Proposition \ref{prop:mainproof}, we have:

$i)$ For any $i \in\{1,\ldots,d\}$,   $\omega_n(i)= \ES[(\tZn^{(1)})_i(\tZn^{(2)})_i]$ satisfies $\forall n\ge n_0$,
$$\omega_{n+1}(i)=\left(1-\gamma_{n+1}\mu_i^\star\right)\frac{n}{n+1} \omega_n(i)+O\left(\frac{\gamma_{n+1}}{n+1}\right)+O(\gamma_{n+1}\ES[|\theta_n|^2|\tZn^{(2)}|]).$$
$ii)$ The following recursion holds  for any $n\ge n_0$,
\begin{align*}
\ES[|\tZnp^{(2)}|^2]&=\left(1-\frac{1}{n+1}\right)^2\ES[|\tZn^{(2)}|^2]+\sum_{i=1}^d \alpha_n^i\omega_n(i)+\frac{\ES[|\theta_n|^2|\tZn^{(2)}|]}{n}\\
&+\frac{{\rm Tr}(\Sigma^\star)}{(n+1)^2}
+\tcb{O\left(\frac{\sqrt{\gamma_n}}{n^2}\vee \frac{1}{n^{3}\gamma_n}\right)},
\end{align*}

where $\alpha_n^i$ is defined in \eqref{eq:alphani} and satisfies $|\alpha_n^i|\lesssim \gamma_n^{-1}n^{-2}$, $i=1,\ldots,d$.
\end{lemma}
\begin{proof}
Set $\underline{\mu}=\min\{\mu_i^\star, i=1,\ldots,d\}>0$. Recall that $n_0\in\mathbb{N}$ is such that $1-\underline{\mu}\gamma_{n}n<0$ for all $n\ge n_0$. For all $n\ge n_0$,  $\Upsilon_n$ and $\Omega_n$ are  well-defined deterministic matrices 
and since for a given $\mu>0$,  $\epsilon_{\mu,n}\sim (n\gamma_n)^{-1}$ and $\epsilon_{\mu,n}-\epsilon_{\mu,n+1}=O(n^{-2}\gamma_n^{-1})$ (see \eqref{eq:contdifepsss}),
we have 
\begin{equation}\label{contnormmatrices}
\gamma_{n+1}\|\Upsilon_n\|=O\left(\frac{1}{n}\right)\quad \textnormal{and}\quad \gamma_{n+1}\|\Omega_n\|=O\left(\frac{1}{n^2}\right).
\end{equation}
Now, let us prove the first statement.\smallskip

\noindent $i)$ Using \eqref{tznpeq}, we have
\begin{align*}
\omega_{n+1}(i)&=\left(1-\gamma_{n+1}\mu_i^\star\right)\left(1-\frac{1}{n+1}\right) \omega_n(i)+O(\gamma_{n+1}\ES[|\theta_n|^2|\tZn^{(2)}|])\\
&+\gamma_{n+1}^2\ES[\tcb{ \left\{ Q \Delta M_{n+1} \right\}_{i}} \tcb{  \left\{ \Upsilon_n Q\Delta M_{n+1} \right\}_i}]+O(\gamma_{n+1} r_n^{(1)})
\end{align*}
where 
$$r_n^{(1)}=\|\Omega_n\|\left(\frac{\ES|\tZn^{(1)}|^2}{\gamma_{n+1}}+ \ES[|\theta_n|^2|\tZn^{(1)}|]\right)+
\|\Upsilon_n\|\left(\ES[ | \tZn^{(1)}|.|\theta_n|^2]+ \gamma_{n+1}\ES|\theta_n|^4\right).$$
The Cauchy-Schwarz inequality, the fact that $|\tZn^{(1)}|=|\theta_n|$ and the consistency condition lead to
$$
\ES[|\theta_n|^2|\tZn^{(1)}|] \leq \left\{ \ES[|\theta_n|^4|\right\}^{1/2} \left\{\ES[|\tZn^{(1)}|^2]\right\}^{1/2} \leq \gamma_{n+1}^{3/2}.
$$
Therefore, \eqref{contnormmatrices} yields:
$$\gamma_{n+1} r_n^{(1)}\lesssim \frac{1}{n^2}\left(1+\gamma_{n+1}^{\frac{3}{2}}\right)+
\frac{1}{n}\left(\gamma_{n+1}^{\frac{3}{2}}+\gamma_{n+1}^3\right)=o\left(\frac{\gamma_n}{n}\right).$$
In the meantime,  under $\mathbf{(H_S)}$ and because $Q \in O_d(\R)$, we have: 
\begin{eqnarray*}
\forall i \in \{1,\ldots,d\} \qquad 
\left| \ES[\left\{ Q \Delta M_{n+1} \right\}_{i} \left\{ \Upsilon_n Q\Delta M_{n+1} \right\}_i] \right| & \lesssim & \|\Upsilon_n\|\ES[|\Delta M_{n+1}|^2] \\
& \lesssim & \|\Upsilon_n\| \ES\|S(\theta_n)\| \\
& \lesssim & \|\Upsilon_n\| (1+\ES|\theta_n|)\\
& \lesssim &  \|\Upsilon_n\| .
\end{eqnarray*}
We therefore deduce from \eqref{contnormmatrices} and from the previous lines  that
$$ \forall i \in \{1,\ldots,d\} \qquad 
\gamma_{n+1}^2  \left| \ES[\left\{ Q \Delta M_{n+1} \right\}_{i} \left\{ \Upsilon_n Q\Delta M_{n+1} \right\}_i] \right|  \lesssim \frac{\gamma_n}{n}.$$

$ii)$ We define $\Delta  N_{n+1}= \Upsilon_n Q \Delta M_{n+1}$ and recall that
$\alpha_n^i$ is defined in \eqref{eq:alphani} by $\alpha_n^i= 2(1-(n+1)^{-1})(\Omega_{n})_{i,i}$. Starting from \eqref{tznpeq} and $|\tZn^{(1)}|=|\theta_n|$, we use  that $\Omega_n$ is a diagonal matrix so that
\begin{align*}
\ES[|\tZnp^{(2)}|^2]&=\left(1-\frac{1}{n+1}\right)^2\ES[|\tZn^{(2)}|^2]+\sum_{i=1}^d \alpha_n^i\omega_n(i)+\gamma_{n+1}^2\ES|\Delta N_{n+1}|^2\\
&+O\left(\gamma_{n+1}\|\Upsilon_n\|\ES[|\theta_n|^2|\tZn^{(2)}|]\right)+O(\gamma_{n+1} r_n^{(2)}),
\end{align*}
where $r_n^{(2)}$ is defined by
$$r_n^{(2)}=\frac{\|\Omega_n\|^2\ES|\theta_n|^2}{\gamma_{n+1}}+ \|\Omega_n\|\|\Upsilon_n\| \ES|\theta_n|^3+\gamma_{n+1}\|\Upsilon_n\|^2\ES|\theta_n|^4.$$
The $(L^4,\sqrt{\gamma_n})$-consistency, the Jensen inequality and \eqref{contnormmatrices} yield
$$ \gamma_{n+1} r_n^{(2)}=O\left(\frac{1}{\gamma_n n^4}+\frac{\sqrt{\gamma_n}}{n^3}+\frac{\gamma_n^2}{n^2}\right)=O\left(\frac{1}{n^3}\right)$$
since $\gamma_n\le c n^{-\frac{1}{2}}$.
To achieve the proof, it remains to show that 
\begin{equation}\label{objectiffinpreuve}
\ES|\Delta N_{n+1}|^2=\frac{{\rm Tr}(\Sigma^\star)}{n^{\tcb{2}}}\tcb{+O\left(\frac{\sqrt{\gamma_n}}{n^2}\vee \frac{1}{n^{3}\gamma_n}\right)}
\end{equation}
First, set $B_n=Q^T \Upsilon_n^2 Q$. Using that $\Upsilon_n$ is a diagonal matrix, we have 
 
\begin{align*}
|\Delta N_{n+1}|^2= {\rm Tr}(|\Delta N_{n+1}|^2) &= {\rm Tr}(\Delta N_{n+1}^T \Delta N_{n+1})\\
&= {\rm Tr}(\Delta M_{n+1}^T  B_n \Delta M_{n+1} )\\
&= {\rm Tr}( B_n \Delta M_{n+1}\Delta M_{n+1}^T  )\\
\end{align*}
Since the trace  is a linear application and $B_n$ is a deterministic matrix, 
\begin{equation}\label{dkjljsl}
\ES[|\Delta N_{n+1}\}|^2{|\cal F}_n]= {\rm Tr}( B_n \ES[\Delta M_{n+1}\Delta M_{n+1}^T|{\cal F}_n])={\rm Tr}( B_n S(\theta_n))
\end{equation}
where we applied Assumption $\mathbf{(H_S)}$. We also have $S(\theta_n)=S(\theta^\star)+O(|\theta_n|)$. For $B_n$, we first remark that 
$$\gamma_{n+1}\Upsilon_n= (n+1)^{-1} \{D^\star\}^{-1}+\Delta_{n+1} $$
where $(\Delta_{n})_{n \geq 0}$ is a sequence of matrices defined by:
$$\Delta_{n} = {\rm Diag}\left\{\frac{1-(n+1)\{\mu_i^\star\}^2\gamma_{n+1}^2}{(n+1) \mu_i^{\star}((n+1) \gamma_{n+1} \mu_i^\star-1)} + \frac{\gamma_{n+1}}{n+1}, \,i=1,\ldots,d\right\}.$$ 
Using that $n\gamma_n^2\rightarrow0$ as $n\rightarrow+\infty$, one easily checks that
$$\|\Delta_n\|\lesssim \frac{1}{n^2\gamma_n}+\frac{\gamma_n}{n}\lesssim \frac{1}{n^2\gamma_n}.$$
As a consequence, 
\begin{eqnarray*}
\gamma_{n+1}^2 B_n & = & Q^T \left\{ \gamma_{n+1} \Upsilon_n\right\}^2 Q \\
& =&  Q^T   \{ (n+1)^{-1} D^\star + \Delta_{n+1}\}^{2} Q \\
 & = & (n+1)^{-2} Q^T \{D^\star\}^{-2} Q + O\left(\frac{1}{n^{\tcb{3}}\gamma_n}\right).
\end{eqnarray*}

It follows from \eqref{dkjljsl} that 
\begin{eqnarray*}
\gamma_{n+1}^2\ES[|\Delta N_{n+1}\}|^2{|\cal F}_n] & =&
\gamma_{n+1}^2 {\rm Tr}( B_n \Delta M_{n+1}\Delta M_{n+1}^T  ) \\
& = & 
\frac{{\rm Tr}(\{ \Lambda^\star\}^{-2} S(\theta^\star))}{(n+1)^2}+O\left(\frac{\ES|\theta_n|}{n^2}\vee \frac{1}{n^{3}\gamma_n}\right)\\
& = & \frac{{\rm Tr}(\{ \Lambda^\star\}^{-2} S(\theta^\star))}{(n+1)^2}+O\left(\frac{\sqrt{\gamma_n}}{n^2}\vee \frac{1}{n^{3}\gamma_n}\right)\\
\end{eqnarray*}
because 
which leads to \eqref{objectiffinpreuve} and achieves the proof.
\end{proof}

\begin{lemma}\label{lem:contechnique2} 
Assume that $(u_n)_{n\ge0}$ is a sequence which satisfies for all $n\ge n_0$ and for a given $\mu>0$:
$$u_{n+1}=\left(1-\gamma_{n+1}\mu\right)\frac{n}{n+1} u_n+\beta_{n+1}$$
with $\beta_n\lesssim \gamma_{n}{n^{-1}}$. Then, 
$u_n=O(n^{-1}).$
\end{lemma}
\begin{proof}  With the convention $\prod_\emptyset=1$ and $\sum_\emptyset=0$, we have for every $n\ge n_0$:
$$u_n=\left(\prod_{k=n_0+1}^n(1-\gamma_k\mu)\frac{k}{k+1}\right)u_{n_0}+\sum_{k=n_0+1}^n \beta_k \prod_{\ell=k+1}^n(1-\gamma_\ell \mu)\frac{\ell}{\ell+1}.$$
Using that for any $x>-1$, $\log(1+x)\le x$, we obtain for every $n\ge n_0+1$
$$\prod_{k=n_0+1}^n(1-\gamma_k\mu)\frac{k}{k+1}\le \frac{n_0}{n+1} e^{-\mu(\Gamma_n-\Gamma_{n_0})}\le C_{n_0} \frac{e^{-\Gamma_n}}{n+1}=O(n^{-1})$$
and,
$$\sum_{k=n_0+1}^n \beta_k \prod_{\ell=k+1}^n(1-\gamma_\ell \mu)\frac{\ell}{\ell+1}\le \frac{1}{n+1}\left(e^{-\mu\Gamma_n}\sum_{k=n_0+1}^n\beta_k(k+1)e^{\mu\Gamma_k}\right).$$
But $\beta_k(k+1)\lesssim\gamma_{k+1}$. Thus, since $x\mapsto x e^{\mu x}$ is increasing on $\ER_+$,
$$\sum_{k=n_0+1}^n\beta_k(k+1)e^{\mu\Gamma_k}\lesssim \sum_{k=n_0+1}^n\gamma_{k+1} e^{\mu\Gamma_k}\le \int_{\Gamma_{n_0+1}}^{\Gamma_{n+1}} e^{\mu x} dx$$
and hence, 
$$\frac{1}{n+1}\left(e^{-\mu\Gamma_n}\sum_{k=n_0+1}^n\beta_k(k+1)e^{\mu\Gamma_k}\right)\le \frac{C_{n_0}}{n+1}.$$
The result follows.
\end{proof}
\begin{rem} \label{rem1} By the expansion $\log(1+x)=x+c(x)x^2$ where $c$ is bounded on $[-1/2,1/2]$, a slight modification of the proof leads to $\liminf_{n\rightarrow+\infty} n u_n>0$ when $\sum\gamma_k^2<+\infty$.
\end{rem}

\begin{lemma}\label{lem:tecnique-clef}
For any sequence $(u_n)_{n \geq 0}$ that satisfies
$$
\forall n \geq 0 \qquad u_{n+1} \leq u_n \left( 1-\frac{1}{n+1}\right)^2 (1+2 n^{-r})+\frac{V}{(n+1)^2}+ \bar{c} n^{-q},
$$
with $r \geq 1$ and $q \geq 2$, then a large enough constant $C$ independent of $n$ exists such that
$$
\forall n \geq 1 \qquad  u_n \leq \frac{V}{n} + C n^{- \{r \wedge (q-1)\}}.
$$
\end{lemma}

\begin{proof}
We establish the result using an induction and denote by $\alpha = r \wedge q$ The statement of the lemma is obvious for $n=1$ by choosing a large enough $C$. Assuming now that the result holds for the integer $n$, we write
\begin{eqnarray*}
u_{n+1} & \leq & \left(\frac{n}{n+1}\right)^2 (1+2 n^{-r}) \left[ \frac{V}{n}+ C n^{- \alpha} \right] + 
\frac{V}{(n+1)^2}+\bar{c} n^{-q} \\
& \leq & V \left[ \frac{n}{(n+1)^2}+\frac{1}{(n+1)^2}\right] + 2 V \left(\frac{n}{n+1}\right)^2 n^{-(r+1)}\\
&& + C n^{-\alpha} \left(\frac{n}{n+1}\right)^2 + 2 C  n^{-(\alpha+r)} \left(\frac{n}{n+1}\right)^2 +\bar{c} n^{-q}\\
& = & \frac{V}{n+1} + C (n+1)^{-\alpha} \mathcal{A}_n
\end{eqnarray*}
where 
$$\mathcal{A}_n :=  \frac{2 V n^{1-r} (n+1)^{\alpha-2}}{C} + \left(\frac{n}{n+1}\right)^{2-\alpha} + 2 \frac{n^{2-\alpha-r} }{(n+1)^{2-\alpha}} + \frac{\bar{c}}{C} n^{-q} (n+1)^{\alpha}
$$
We now choose  $\alpha < 2$ and use the first order approximations:
$$
\mathcal{A}_n = 1-(2-\alpha) n^{-1} + C^{-1} \left[ 2V n^{\alpha-(1+r)} + 2 n^{-r} + \bar{c} n^{\alpha-q}\right] + o(n^{-1}).
$$
Then, a large enough $C$ exists such that $\mathcal{A}_n \leq 1$ for any $n \geq 1$ as soon as the powers of $n$ are lower than $1$ inside the brackets on the right hand side of the equality above. Hence, $\alpha$ should be chosen such that
$ \{(1+r)-\alpha\} \wedge \{r\} \wedge \{\alpha - q\} \geq 1$ and the largest possible value of $\alpha$ corresponds to 
the choice $$
\alpha = (q-1) \wedge r.
$$
For such a choice, a large enough $C$ exists such that the recursion holds, which ends the proof of Lemma \ref{lem:tecnique-clef}.
\end{proof}

\section{Growth at infinity  under the KL gradient inequality}
\label{sec:appendixAA}
In this section, we prove the property \eqref{eq:polynincrKL} of Proposition \ref{lemma:KLimphi}.
Without loss of generality, we can assume that $\theta^\star=f(\theta^\star)=0$. 
\smallskip
\begin{proof}
Consider $0 \leq t \leq s$ and $x \in \mathbb{R}^d$. We then associate the solution of the differential equation associated to the flow $-\nabla f$ initialized at $x$:
$$
\chi_x(0)=x \qquad \text{and} \qquad  \dot \chi_x = - \nabla f(\chi_x).
$$
The length of the curve $L(\chi_x,t,s)$ is defined by
$$
L(\chi_x,t,s)= \int_{t}^s \|\dot \chi_x(\tau) \| d \tau.
$$
Under Assumption $\mathbf{(H_{KL}^r)}$, we can consider $\varphi(a)=\frac{a^{1-r}}{1-r}$  and we have that
$$
\varphi'(f(x)) \|\nabla f(x)\| \geq m > 0.
$$
We now observe that $e: s \longmapsto \varphi(f(\chi_x(s)))$ satisfies:
\begin{eqnarray*}
e'(\tau) &=& \varphi'(f(\chi_x(\tau))) \langle \nabla f(\chi_x(\tau)),\dot \chi_x(\tau)\rangle \\
& =& - \varphi'(f(\chi_x(\tau))) \|\nabla f(\chi_x(\tau))\|^2\\
&  \leq &- m \|\dot \chi_x(\tau)\|
\end{eqnarray*}
We deduce that:
\begin{equation}\label{eq:minolength}
e(t)-e(s) = \int_{s}^t e'(\tau) d\tau \geq m \int_{t}^s  \|\dot \chi_x(\tau)\| d\tau \geq m L(\chi_x,t,s)
\end{equation}
Now choosing $t=0$ and $s \longrightarrow + \infty$, we have  
$
e(0)-\lim_{s \longrightarrow + \infty} e(s) = \varphi(f(x)) - \varphi(\min f) = \varphi(f(x)),
$
and Equation \eqref{eq:minolength} yields
$$
\varphi(f(x)) \geq m L(\chi_x,0,+\infty) \geq m \|x\|
$$
because $\chi_x(+\infty)=\arg\min f=0$. We deduce that
$$
f(x) \geq \varphi^{-1} (m \|x\|) = \left\{ m (1-r)\right\}^{\frac{1}{1-r}} \|x\|^{\frac{1}{1-r}}.
$$
which is the desired conclusion.
\end{proof}

\end{document}